\documentclass[12pt,a4paper,oneside,reqno,english]{amsart}
\usepackage[T1]{fontenc}
\usepackage[utf8]{inputenc}
\usepackage{verbatim}
\usepackage{amsmath}
\usepackage{amstext}
\usepackage{amsthm}
\usepackage{stmaryrd}
\usepackage{setspace}
\usepackage{graphicx}
\makeatletter

\pdfpageheight\paperheight
\pdfpagewidth\paperwidth

\numberwithin{equation}{section}
\numberwithin{figure}{section}
\theoremstyle{plain}
\newtheorem{thm}{\protect\theoremname}[section]
\theoremstyle{remark}
\newtheorem{rem}[thm]{\protect\remarkname}
\theoremstyle{plain}
\newtheorem{lem}[thm]{\protect\lemmaname}
\theoremstyle{definition}
\newtheorem{problem}{\protect\problemname}[section]

\usepackage{amsfonts}
\usepackage{amsthm}
\usepackage{epsfig}
\usepackage{mathrsfs}

\@ifundefined{definecolor}
 {\usepackage{color}}{}

\makeatletter

\newcommand{\Rmnum}[1]{\expandafter\@slowromancap\romannumeral#1@}\makeatother

\numberwithin{equation}{section}

\allowdisplaybreaks

\newcommand{\norm}[1]{\left\Vert#1\right\Vert}

\newcommand{\set}[1]{\left\{#1\right\}}

\newcommand{\pr}[1]{\left(#1\right)}

\newcommand{\defs}{:=}

\newcommand{\me}{\mathrm{e}}
\newcommand{\dif}{\mathrm{d}}

\DeclareSymbolFont{lettersA}{U}{pxmia}{m}{it}
\DeclareMathSymbol{\piup}{\mathord}{lettersA}{"19}

\newcommand{\Real}{\mathbb R}

\newcommand{\mr}[1]{\mathrm{#1}}

\newcommand{\mf}[1]{\mathscr{#1}}

\newcommand{\mcc}{\mathcal{C}}

\makeatother

\usepackage{babel}
\providecommand{\lemmaname}{Lemma}
\providecommand{\problemname}{Problem}
\providecommand{\remarkname}{Remark}
\providecommand{\theoremname}{Theorem}

\begin{document}
\title[Steady Shocks for Reacting Euler Flows]{Transonic Shocks for 2-D Steady Exothermically Reacting Euler Flows in a Finite Nozzle}

\author{Beixiang Fang}
\author{Piye Sun}
\author{Qin Zhao}

\address{B.X. Fang: School of Mathematical Sciences, MOE-LSC, and SHL-MAC,
Shanghai Jiao Tong University, Shanghai 200240, China }
\email{\texttt{bxfang@sjtu.edu.cn}}

\address{P.Y. Sun: School of Mathematical Sciences,
	Shanghai Jiao Tong University, Shanghai 200240, China }
\email{\texttt{sunpiye@amss.ac.cn}}

\address{Q. Zhao: Department of Mathematics, School of Science, Wuhan University
of Technology, Wuhan, 430070, China}
\email{\texttt{qzhao@whut.edu.cn}}


\date{\today}

\begin{abstract}
This paper concerns the existence of transonic shocks for steady exothermically reacting Euler flows in an almost flat nozzle with the small rate of the exothermic reaction. One of the key points is to quantitatively determine the position of the shock front in the nozzle.
We focus on the contributions of the perturbation of the flat nozzle and the exothermic reaction in determining the position of the shock front by comparing the orders of $\sigma$ and $\kappa$, where $\sigma$ represents the scale of the perturbation of the flat nozzle and $\kappa$ the rate of the exothermic reaction.
To this end, a free boundary problem for the linearized reacting Euler system is proposed to catch an approximating position of the shock front as well as the associated approximating shock solution, with which the existence of a shock solution close to it can be established via a nonlinear iteration scheme.
One of the key steps is solving the nonlinear equation derived from the solvability condition for the elliptic sub-problem in the domain of the subsonic flow behind the shock front, which determines the free boundary of the proposed problem.
Four typical cases are analyzed which describe possible interactions between the geometry of the nozzle boundary and the exothermic reaction.
The results also manifest that exothermic reaction has a stabilization effect on transonic shocks in the nozzles.
\end{abstract}

\maketitle
\tableofcontents

\section{Introduction}
In this paper we are concerned with the transonic shocks for 2-D steady exothermically reacting Euler flows in an almost flat nozzle.
Assume that a supersonic flow enters a 2-D nozzle, in which exothermic reaction occurs, and leaves it with a relatively high pressure, then it is expected that a shock front appears in the flow field such that the pressure rises to coincide with the value at the exit.
Then catching the position of the shock front is one of the most important ingredients in determining the flow field within the nozzle.
For gas flows without exothermic reactions, Courant and Friedrichs have pointed out in \cite{CourantFriedrichs1948} that, in order to determine the position of the shock front, additional conditions are needed at the exit of the nozzle and they propose to use the pressure condition.
This fact has been verified by rigorous mathematical analysis, for instance, see \cite{Chen_S2009CMP,FangXin2021CPAM,LiXinYin2013ARMA} and references therein.
As the exothermic reaction is involved in the flow field, it is natural to ask, in catching the position of the shock front, whether the pressure condition at the exit is still sufficient or not, and what the effects the exothermic reaction brings.
This paper is devoted to study these problems.

\subsection{Mathematical Formulation of the Problem.}

Let $ \varphi_w(x_1) $ be a smooth function and
\begin{equation*}
	\mathscr{D} = \{ (x_1,x_2): 0 \leq x_1 \leq L, 0 \leq x_2 \leq \varphi_w(x_1) \}
\end{equation*}
be the domain bounded by the nozzle (See Figure \ref{fig:domain}).
Denote the entrance, the exit, the lower boundary and the upper boundary of the nozzle by, respectively,
\begin{equation*}
	\begin{split}
		E_1 &= \{(x_1,x_2): x_1=0,x_2 \in [0,1]\},\\	
		E_2 &= \{(x_1,x_2): x_1=L,x_2 \in [0, \varphi_w (L)] \},\\
		W_1 &= \{(x_1,x_2): x_1 \in [0,L], x_2=0 \},\\
		W_2 &= \{(x_1,x_2): x_1 \in [0,L], x_2=\varphi_w(x_1) \}.
	\end{split}
\end{equation*}
\begin{figure}[htbp]
	\centering
	\includegraphics[height=5cm, width=11cm]{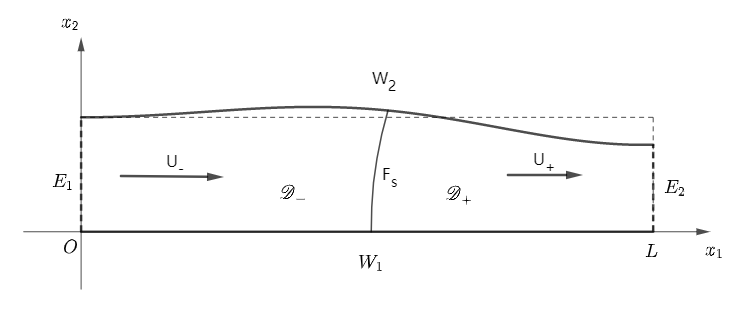}
	\caption{A transonic flow with a single shock in an almost flat nozzle.}
	\label{fig:domain}
\end{figure}

The steady exothermically reacting flow in the nozzle is governed by the two-dimensional Euler system with a combustion term:
\begin{equation}\label{Euler}
	\begin{cases}
		\partial_{x_1}(\rho u_1) + \partial_{x_2}(\rho u_2)=0,
		\\\partial_{x_1} (\rho u_1^2 +p) + \partial_{x_2}( \rho u_1 u_2)=0,
		\\ \partial_{x_1}(\rho u_1 u_2) + \partial_{x_2} ( \rho u_2^2 +p)=0,
		\\ \partial_{x_1}\left(\rho u_1 \left(E+\displaystyle\frac{p}{\rho}\right)\right)+\partial_{x_2}\left(\rho u_2 \left(E+\displaystyle\frac{p}{\rho}\right)\right)=0,
		\\ \partial_{x_1}(\rho u_1 Z) + \partial_{x_2}(\rho u_2 Z)=-\kappa \rho \phi(T)Z.
	\end{cases}
\end{equation}
Here $(u_1, u_2)$ represents the horizontal and vertical component of the velocity
and $(p, \rho, T)$ stands for the pressure, density and temperature, respectively. $Z \in [0,1]$ represents the fraction of unburned gas in the mixture.
$\kappa>0$ is the rate of the exothermic reaction.
Typically, the ignition function $\phi (T)$ has the Arrhenius form,
\begin{align*}
	\phi(T)=
	\begin{cases}
		0, &\text{if $T<T_0$}, \\
		T^{a}\me^{-\frac{\mathcal{E}}{\mathcal{R}_0 (T-T_0)}}, &\text{if $T \geq T_0$},
	\end{cases}
\end{align*}
where $a$ is a positive constant, $\mathcal{R}_0$ is the universal gas constant,  $T_0$ is the ignition temperature and $\mathcal{E}$ is activation energy for the gas.
$E$ represents the total energy
\begin{equation}
	E=\frac{1}{2}(u_1^2+u_2^2)+e+{q_{\mr{e}}} Z,
\end{equation}
where $e$ is the internal energy of the gas, and the constant ${q_{\mr{e}}}$ is the binding energy of unburned gas. In particular, for polytropic gases with the adiabatic exponent $\gamma$, the pressure and the internal energy can be represented by
\begin{equation}
	p=\rho \mathcal{R} T=(\gamma-1) \me ^{\frac{S}{c_{\mr{v}}}}\rho^\gamma, \qquad\text{or }\qquad e=\frac{1}{\gamma-1} \frac{p}{\rho},
\end{equation}
where $S$ is the entropy, $c_{\mr{v}}$ is the specific heat at constant volume and $\mathcal{R}$ is a constant which is the universal gas constant $\mathcal{R}_0$ divided by the effective molecular weight of the gas. Then the sound speed is given by $c=\sqrt{{\gamma p}/{\rho}}$.
In this paper we shall choose $U=(p,\theta,q,S,Z)^\top$ to represent the independent parameters of the fluid,
where $\theta=\arctan \displaystyle\frac{u_2}{u_1}$ and $q=\sqrt{u_1^2+u_2^2}$ stands for the velocity direction and the speed of the fluid. Recall that the flow is supersonic in case that the Mach number $M\defs q/c>1$, and it is subsonic if $M<1$.

In this paper, a supersonic flow will be given at the entrance of the nozzle, in which exothermic reaction occurs. And the goal is to clarify whether or not there exists a transonic shock front across which the flow will become subsonic downstream with a given pressure at the exit.
As the notation shown in Figure \ref{fig:domain}, the subscript ``$ - $'' will represent the parameters of the flow ahead of the shock front and the subscript ``$ + $'' behind the shock front.
For a shock front at the position $x_1=\varphi_s(x_2)$, the domain $\mathscr{D}$ is seperated into two parts: the upstream domain
\begin{align*}
	\mathscr{D}_- = \{(x_1,x_2): 0 < x_1 < \varphi_s(x_2), 0 < x_2 < \varphi_w(x_1) \},
\end{align*}
and the downstream domain
\begin{align*}
	\mathscr{D}_+ &= \{(x_1,x_2): \varphi_s(x_2) < x_1 < L, 0 < x_2 < \varphi_w(x_1) \}.
\end{align*}
Let $U_-=(p_-,\theta_-,q_-,S_-,Z_-)^\top$ and $U_+=(p_+,\theta_+,q_+,S_+,Z_+)^\top$ be the state of the upstream and downstream flow in $\mathscr{D}_-$ and $\mathscr{D}_+$, respectively.
Since a discontinuous jump will occur at the shock front, the following Rankine-Hugoniot conditions (which will be abbreviated as R-H conditions) should be satisfied on the shock:
\begin{equation}\label{eq:RH}
	\begin{cases}
		[\rho u_1]= \varphi_s'(x_2) [\rho u_2],
		\\ [\rho u_1^2 +p] = \varphi_s'(x_2) [\rho u_1 u_2],
		\\ [\rho u_1 u_2]=\varphi_s'(x_2) [\rho u_2^2 +p],
		\\ \left[\rho u_1 \left(E+\displaystyle\frac{p}{\rho}\right)\right]=\varphi_s'(x_2) \left[ \rho u_2 \left(E+\displaystyle\frac{p}{\rho}\right) \right],
		\\ [\rho u_1 Z]=\varphi_s'(x_2) [\rho u_2 Z],
	\end{cases}
\end{equation}
where $\displaystyle \varphi_s'(x_2)=\frac{\dif \varphi_s}{\dif x_2}(x_2)$ and $[\cdot]$ stands for the jump of the corresponding quantity across the shock front, i.e. $[w]=w_+ - w_-$.

Finally, on the nozzle boundaries $W_1$ and $W_2$, it is assumed that the fluid cannot penetrate them, which yields the following slip boundary conditions:
\begin{equation}\label{eq:slip_bdry_cond}
	\begin{aligned}
		& \theta_{\pm} = 0,& &\text{ on }W_1 \cap \overline{\mathscr{D}_\pm},\\
		& \tan \theta_{\pm}  = \varphi_w'(x_1),& &\text{ on }W_2 \cap \overline{\mathscr{D}_\pm}.
	\end{aligned}
\end{equation}

Hence, the problem for the existence of the transonic shock for the steady exothermic reacting flow in a nozzle could be formulated as the following free boundary problem.

\begin{problem}\label{pr:SP} 
	Given a supersonic state $ U_-=U_{\mr{in}}(x_2) $ at the entrance $ E_{1} $, and a relatively high pressure $ p_+=P_{\mr{out}}(x_2) $ at the exit $ E_{2} $, whether or not there exists a shock solution $ (U_-, U_+) $ in $ \mf{D} $ to the 2-D steady reacting Euler system \eqref{Euler} for the  exothermically reacting flow, with the position of the shock front being
	\begin{equation*}
		F_s\defs\set{(x_1,x_2)\in\Real^2: x_1 = \varphi_s(x_2), x_2 \in [0,Y_s]},
	\end{equation*}
	where $\varphi_s(x_2)$ is an unknown function and $(Y_s, \varphi_s (Y_s))$ is the intersection of the shock-front and the upper boundary, i.e. $Y_s=\varphi_w(\varphi_s(Y_s))$,
	such that the R-H conditions \eqref{eq:RH} are satisfied on $ F_s $, and the boundary condition \eqref{eq:slip_bdry_cond} holds on $W_1$ and $W_2$ (see Figure \ref{fig:domain}).
\end{problem}

The reacting Euler system \eqref{Euler} is a nonlinear system of mixed type partial differential equations which is hyperbolic in the supersonic region $\mathscr{D}_-$ and elliptic-hyperbolic composite in the subsonic region $\mathscr{D}_+$. Moreover, the location of the shock front is unknown so that $\mathscr{D}_-$ and $\mathscr{D}_+$ are both undetermined domains connected by a free boundary.
Not only the state of incoming flow and the pressure at the exit, the shape of the boundary of the nozzle and exothermic reactions can also influence the location of the shock front.
In this paper, we focus on perturbation problems which involve a small rate of exothermic reactions and a small perturbation of the boundary of a flat nozzle. Hence, firstly we introduce the steady normal shock solutions without exothermic reactions in a flat nozzle as the background solution.

Let $\mathscr{D}_0$ be a flat nozzle with length $L$ and height $1$:
\begin{equation*}
	\mathscr{D}_0=\{ (x_1,x_2): 0 \leq x_1 \leq L, 0 \leq x_2 \leq 1 \}.
\end{equation*}
For a given uniform supersonic state $\bar{U}_- = (\bar{p}_-, 0, \bar{q}_-, \bar{S}_-, \bar{Z})^\top$ with the temperature $\bar{T}_->T_0$ ,
it can be easily verified that there exists a unique subsonic state $\bar{U}_+ = (\bar{p}_+, 0, \bar{q}_+, \bar{S}_+, \bar{Z})^\top$ such that
\begin{equation}\label{RH}
	\begin{cases}
		[\bar{\rho} \bar{q}]=0,
		\\ [\bar{\rho} \bar{q}^2 + \bar{p}]=0,
		\\ \left[\bar{E}+\displaystyle\frac{\bar{p}}{\bar{\rho}}\right] =0,
		\\ [\bar{Z}]=0.
	\end{cases}
\end{equation}
That is, $ \bar{U}_+ $ can be connected to $\bar{U}_-$ through a plane normal shock.

\begin{rem}
 If the temperature $\bar{T}_-$ of the incoming flow is lower than the ignition temperature $T_0$, the gas ahead of the shock front will be unburnt. Since the temperature increases across the shock, if the temperature $\bar{T}_+$ across the shock front is greater than the ignition temperature $T_0$, combustion reactions will take place behind the shock front. The combustion process with its precompression shock wave is usually called a detonation wave(see \cite{CourantFriedrichs1948}, Section 92). For this model,  similar results can be established by the same arguments in this paper. For simplicity of presentation, we focus on the case that the combustion reaction occurs in the whole nozzle.
\end{rem}

Let $ U_{\mr{in}}(x_2)\equiv\bar{U}_- $ and $ P_{\mr{out}}(x_2)\equiv\bar{p}_+ $. Then, for any $\bar{x}_s\in [0,L]$, as $ \bar\varphi_w(x_1)\equiv 1 $ and $ \kappa = 0 $, $ \pr{\bar{U}_-;\ \bar{U}_+;\ \bar{x}_s} $ consists a steady plane normal shock solution to Problem \ref{pr:SP} in the sense that
\begin{equation}\label{eq:normal-shock}
	\bar{U}(x_1,x_2)\defs
	\begin{cases}
		\bar{U}_-, & \text{ for }0 < x < \bar{x}_s,\, 0 < y < 1,\\
		\bar{U}_+, & \text{ for }\bar{x}_s < x < L,\, 0 < y < 1.
	\end{cases}
\end{equation}

Without loss of generality, we may assume
\begin{equation}
	\bar{\rho}_- \bar{q}_- = \bar{\rho}_+ \bar{q}_+ =1.
\end{equation}

Based on the steady normal shock solution \eqref{eq:normal-shock}, this paper is going to solve Problem \ref{pr:SP} with small perturbed boundary data and sufficiently small $ \kappa>0 $.

Assume that $\varphi_w (x_1)$ is a small perturbation of $\bar{\varphi}_w (x_1)$ with the form:
\begin{equation}\label{perturb2}
	\varphi_w (x_1) := 1 + \int_0^{x_1} \tan (\sigma\Theta(s)) ds,
\end{equation}
and the pressure $ P_{\mr{out}}(x_2) $ at the exit is a small perturbation of $ \bar{p}_+ $ with the form:
\begin{equation}\label{perturb1}
	P_{\mr{out}} (x_2) := \bar{p}_+ + P_\mr{e}(x_2; \sigma, \kappa),
\end{equation}
where $ \sigma, \kappa>0 $ are sufficiently small, $P_{\mr{e}}(\cdot ; \sigma, \kappa) \in C^{2,\alpha}(\overline{\mathbb{R}_+})$,  $\Theta \in C^{2,\alpha}[0,L]$ are given functions with $0 < \alpha < 1$, satisfying
\begin{equation}\label{perturb3}
	\norm{P_\mr{e}(\cdot ; \sigma, \kappa)}_{C^{2,\alpha}(\overline{\mathbb{R}_+})} < + \infty,  \qquad
	\| \Theta \| _{C^{2,\alpha}[0,L]} < + \infty,
\end{equation}
\begin{equation}\label{perturb4}
	\Theta(0)=\Theta'(0)=\Theta''(0)=0.
\end{equation}

In this paper, we are going to establish the existence of the shock solution to Problem \ref{pr:SP} with the detailed boundary data as below.

\begin{problem}\label{pr:SmallPerturbation}
	Assume that the upper boundary of the nozzle is described by $ \varphi_w (x_1)$ in \eqref{perturb2}. Let $ U_{\mr{in}}(x_2)\equiv\bar{U}_- $ and $ P_{\mr{out}}(x_2) $ be given as \eqref{perturb1}.
	Then try to determine a transonic shock solution $ \pr{ U_-(x_1,x_2);\ {U}_+(x_1,x_2);\ \varphi_s(x_2)} $ (see Figure \ref{fig:domain}) to Problem \ref{pr:SP} in the sense that:
	\begin{enumerate}
		\item The location of the shock-front is
		\begin{equation}
			F_s = \{(x_1,x_2): x_1 = \varphi_s(x_2), x_2 \in [0,Y_s] \},
		\end{equation}
		where $(Y_s, \varphi_s (Y_s))$ is the intersection of the shock-front and the upper boundary, i.e. $Y_s=\varphi_w(\varphi_s(Y_s))$.
		Then the domain $\mathscr{D}$ is separated by $ F_s $ into two parts:
		\begin{equation}
			\begin{split}
				\mathscr{D}_- &= \{(x_1,x_2): 0 < x_1 < \varphi_s(x_2), 0 < x_2 < \varphi_s(x_1) \},\\
				\mathscr{D}_+ &= \{(x_1,x_2): \varphi_s(x_2) < x_1 < L, 0 < x_2 < \varphi_s(x_1) \},
			\end{split}		
		\end{equation}
		where $ \mf{D}_- $ is the region of the supersonic flow ahead of the shock front, and $ \mf{D}_+ $ is the region of the subsonic flow behind it.
		\item $U_-(x_1,x_2)$ and $U_+(x_1,x_2)$ satisfy the reacting Euler system \eqref{Euler} in the sense of classical solutions in the domain $\mathscr{D}_-$ and $\mathscr{D}_+$ respectively.
		\item The state at the entrance coincides with the uniform supersonic state:
		\begin{equation}
			U_- = \bar{U}_-, \quad \text{on $E_1$},
		\end{equation}
		and the pressure at the exit is the given receiver pressure:
		\begin{equation}
			p_+ = \bar{p}_+ + P_\mr{e}(x_2; \sigma, \kappa), \quad \text{on $E_2$}.
		\end{equation}
		\item The slip condition holds along the boundary $W_1$ and $W_2$:
		\begin{align}
			&\theta_-(x_1) = 0, & \text{on } &W_1 \cap \overline{\mathscr{D}_-}, \\
			&\theta_-(x_1) = \sigma\Theta(x_1), &\text{on } &W_2 \cap \overline{\mathscr{D}_-}, \\
			&\theta_+(x_1) = 0, &\text{on } &W_1 \cap \overline{\mathscr{D}_+}, \\
			&\theta_+(x_1) = \sigma\Theta(x_1), &\text{on } &W_2 \cap \overline{\mathscr{D}_+}.
		\end{align}
		\item	The R-H conditions \eqref{eq:RH} hold along the shock-front $F_s$.
	\end{enumerate}
\end{problem}

Notice that the position of the normal steady shock $\bar{x}_s$ in \eqref{eq:normal-shock} can be anywhere in $(0,L)$ so that the background solutions \eqref{eq:normal-shock} cannot provide any information for the location of the shock front in Problem \ref{pr:SmallPerturbation}. Hence one of the key difficulties is to determine the location of the shock front.
Motivated by Fang and Xin's work \cite{FangXin2021CPAM}, a linearized problem will be introduced first and the position of an approximating shock front can be obtained by the solvability of an elliptic sub-problem in the subsonic region(see \cite{FangXin2021CPAM,Grisvard1985}).
The pressure at the exit $P_\mr{out}(x_2;\sigma, \kappa)$ should be chosen properly such that the shocks exist for small $\sigma$ and $\kappa$.
For explicitness of the argument, we assume that $P_\mr{e}$ has the form
\begin{align}
	P_\mr{e}(x_2;\sigma, \kappa) = \sigma P_\sigma (x_2) + \kappa P_\kappa (x_2),
\end{align}
and four typical cases will be considered:
\begin{enumerate}
	\item the perturbation of the boundary has the main effects compared with the exothermic reaction: $\kappa= A_1 \sigma^s$ with $s>1$ and $A_1>0$. Then $P_\mr{e}$ has the form $P_\mr{e}(x_2;\sigma, \kappa) = \sigma P_\sigma(x_2)+ \sigma^s A_1 P_\kappa(x_2)$.
	\item the exothermic reaction has the main effects compared with the perturbation of the boundary: $\sigma= A_2 \kappa^s$ with $s>1$ and $A_2>0$. Then $P_\mr{e}$ has the form $P_\mr{e}(x_2;\sigma, \kappa) = \kappa P_\kappa(x_2)+ \kappa^s A_2 P_\sigma(x_2)$.
	\item the influence of the perturbation of the boundary and the exothermic reaction is at the same level: $\sigma=A \kappa$ for some constant $A>0$ so that $P_\mr{e}$ can be rewritten into $P_\mr{e}(x_2;\sigma, \kappa)=\kappa P_A(x_2)$, where $P_A(x_2)= A P_\sigma(x_2) + P_\kappa(x_2)$.
	\item $\big(P_{\sigma}(x_2), P_{\kappa}(x_2)\big)$ satisfies a special condition such that the initial approximating location of the shock front is independent of $\sigma$ and $\kappa$.
\end{enumerate}

For each case, the range of a proper pressure is analysed. Then with a proper pressure, an approximating shock solution can be solved from the linearized problem.
Based on the approximation of the shock solution, a nonlinear iteration scheme will be designed and executed to approach a shock solution to Problem \ref{pr:SmallPerturbation}.
In this paper, the existence of the shock solution to Problem \ref{pr:SmallPerturbation} will be established by showing the following theorem.
\begin{thm}
	Given $\Theta$ and a proper pressure $P_\mr{e}(\cdot; \sigma, \kappa)$ at the exit satisfying one of Perturbation Hypotheses defined in Section \ref{section: main theorem}, there exists a sufficiently small constant $\epsilon_0>0$, such that for any positive $\kappa$ and $\sigma$ satisfying $\kappa+\sigma < \epsilon_0$, there exists a transonic shock solution $(U_-;U_+;\varphi_s)$ to Problem \ref{pr:SmallPerturbation} such that $(U_-;U_+)$ is a small perturbation of $(\bar{U}_-, \bar{U}_+)$.
\end{thm}
The detailed version of the main theorem will be stated in Theorem \ref{Main Theorem}.

\subsection{Related Literatures.}

It is well-known that the study of gas flows in nozzles plays a fundamental role in aircraft engines, wind tunnels, rockets, etc.
For flows of general fluids such as polytropic gases, Courant and Friedrichs first gave a systematic analysis in \cite{CourantFriedrichs1948} from the viewpoint of nonlinear partial differential equations.
As shocks occur in the flow field, they point out that, the position of the shock front cannot be determined unless additional conditions are imposed at the exit and the pressure condition is suggested and preferred(see \cite[Page 373-374]{CourantFriedrichs1948}). Since then, in order to establish a rigorous mathematical theory for flows with shocks in a nozzle, various nonlinear PDE models and different boundary conditions have been proposed, fruitful ideas and methods had been developed, and substantial progresses had been made, for instance, see \cite{BaeFeldman2011ARMA, ChenFeldman2003JAMS, Chen_S2005TAMS, Chen_S2009CMP, ChenYuan2008ARMA, FangGao2021JDE, FangLiuYuan2013ARMA, FangXin2021CPAM, LiXinYin2010JDE, LiXinYin2013ARMA, LiuXuYuan2016AdM, ParkRyu_2019arxiv, ParkY_2021ARMA, WengXieXin2021SIMA, XinYanYin2011ARMA, XinYin2005CPAM, YuanZhao2020AcMath} and reference therein. In particular,
two typical kinds of nozzles are studied. One is an expanding nozzle of an angular sector or a diverging cone.
In \cite{CourantFriedrichs1948}, Courant and Friedrichs established the unique existence of a transonic shock solution in such a nozzle with given constant pressure at the exit. Based on this shock solution, the well-posedness of shock solutions in an expanding nozzle has been established in \cite{Chen_S2009CMP, LiXinYin2009CMP, LiXinYin2013ARMA}, with prescribed pressure at the exit as suggested by Courant and Friedrichs. The other is a flat nozzle with two parallel walls.
In this case, the existence of planar normal shock solutions can be easily established. However, the position of the shock front cannot be determined since it can be arbitrary in the flat nozzle. An idea to deal with this difficulty is presuming that the shock front goes through a fixed point which is given in advance artificially, and spontaneously replacing the pressure condition at the exit by other conditions, for instance, see \cite{ChenFeldman2003JAMS,Chen_S2005TAMS, XinYanYin2011ARMA,XinYin2005CPAM}.
Recently, in \cite{FangXin2021CPAM}, Fang and Xin proposed another idea to deal with this difficulty and successfully obtain the position of the shock front with the pressure condition at the exit, as suggested by Courant and Friedrichs. See also \cite{FangGao2021JDE, FangGao2022SIMA} for more applications of the idea.
Based on the methods and results developed in these literature, this paper studies the nozzle shock problem for steady exothermically reacting Euler flows.
One may refer to, for instance, \cite{CKZ2017, ChenWagner2003, XZZ2018, Zumbrun2017} for studies on weak solutions and detonation fronts of the exothermically reacting Euler system.

\subsection{Organization of the Paper.}

In Section 2, Problem \ref{pr:SmallPerturbation} is reformulated under the Lagrange transformation first.
Then the free boundary problem based on the background solution $(\bar{U}_-, \bar{U}_+)$ is introduced, whose solution gives the initial approximating position of the shock front as well as the associated approximating shock solution.
Main theorems are described in the last part of Section 2.
In Section 3, the free boundary problem introduced in Section 2 is solved and the existence of the solution is established.
Four typical cases are analyzed such that there exists a solution to the nonlinear equation derived from the solvability condition for the elliptic sub-problem for the downstream flows behind the shock front.
In Section 4, a nonlinear iteration scheme is designed and carried out which converges to a shock solution of the reacting Euler system close to the approximation obtained in Section 3.

\section{Lagrange Transformation and Main Theorems}

Since there are conserved quantities along streamlines in the Euler system, the Lagrange transformation, which straightens the streamlines, is often used to simplify the equations (cf. \cite{Chen_S2005TAMS,CourantFriedrichs1948}). Meanwhile, the perturbed boundary will become flat via Lagrange transformation because it is a streamline under the slip boundary condition. Hence Lagrange transformation will also be employed for the reacting Euler system in this paper.

\subsection{Reformulation of the Problem under Lagrange Transformation.}

From the first equation of \eqref{Euler}, there exists a function $\psi(x)$ such that $\nabla \psi = (-\rho u_2 , \rho u_1)$ and $\psi(0,0)=0$. Then the Lagrange transformation can be defined as
\begin{equation}\label{LTrans}
\begin{cases}
y_1 = x_1, \\
y_2 = \psi(x_1, x_2) := \int_{(0,0)}^{(x_1,x_2)}  \rho u_1(s,t) \dif t -\rho u_2(s,t)\dif s .
\end{cases}
\end{equation}
Under the transformation, the reacting Euler equations \eqref{Euler} becomes
\begin{align}
\label{LEuler1} &\partial_{y_1}\left(\frac{1}{\rho u_1}\right) - \partial_{y_2}\left(\frac{u_2}{u_1}\right)=0,\\
\label{LEuler2} &\partial_{y_1}\left(u_1 + \frac{p}{\rho u_1}\right) - \partial_{y_2}\left(\frac{p u_2}{u_1}\right)=0,\\
\label{LEuler3} &\partial_{y_1} u_2 + \partial_{y_2} p=0,\\
\label{LEuler4} &\partial_{y_1} \left( \frac{1}{2} q^2 +\frac{\gamma p}{(\gamma-1)\rho}+ {q_{\mr{e}}} Z \right)=0,\\
\label{LEuler5} &\partial_{y_1} Z=-\kappa \frac{\phi(T)}{u_1} Z.
\end{align}
The shock front $F_s$ becomes
\begin{equation}
\Gamma_s=\{ (y_1,y_2): y_1=\psi_s(y_2), 0<y_2<1 \}
\end{equation}
and R-H conditions become
\begin{align}
\label{LRH1} &\left[\frac{1}{\rho u_1}\right]
+ \psi'_s(y_2) \left[\frac{u_2}{u_1}\right] =0, \\
\label{LRH2} &\left[u_1 + \frac{p}{\rho u_1}\right]
+ \psi'_s(y_2) \left[\frac{p u_2}{u_1}\right]=0, \\
\label{LRH3} &\left[u_2\right]-\psi'_s(y_2) \left[p\right]=0, \\
\label{LRH4} &\left[\frac{1}{2} q^2 + \frac{\gamma p}{(\gamma-1)\rho}\right]=0, \\
\label{LRH5} &\left[ Z \right]=0.
\end{align}

Eliminating $\psi_s'$ in the R-H conditions yields
\begin{align}
\label{G1}G_1(U_+, U_-) &:= \left[\frac{1}{\rho u_1}\right][p]+ \left[\frac{u_2}{u_1}\right][u_2]=0, \\
\label{G2}G_2(U_+, U_-) &:= \left[u_1 + \frac{p}{\rho u_1}\right][p]+
\left[\frac{p u_2}{u_1}\right][u_2]=0, \\
\label{G3}G_3(U_+, U_-) &:= \left[\frac{1}{2} q^2 + \frac{\gamma p}{(\gamma-1)\rho}\right] =0, \\
\label{G4}G_4(U_+, U_-) &:= [Z]=0,
\end{align}
and \eqref{LRH3} can be rewritten as
\begin{equation}
\label{G5}G_5(U_+, U_- ; \psi'_s) := [u_2] - \psi_s' [p] =0.
\end{equation}
Then on the shock front, R-H conditions \eqref{LRH1}-\eqref{LRH5} can be replaced by conditions \eqref{G1}-\eqref{G5}.

Under the Lagrange transformation, the entrance $E_1$ and the exit $E_2$ become
\begin{align}
&\Gamma_1 = \{(y_1,y_2): y_1=0, 0<y_2<1 \}, \\
&\Gamma_3 = \{(y_1,y_2): y_1=L, 0<y_2<1 \},
\end{align}
and the lower wall $W_1$ and the upper wall $W_2$ become
\begin{align}
&\Gamma_2 = \{(y_1,y_2): 0<y_1<L, y_2=0 \}, \\
&\Gamma_4 = \{(y_1,y_2): 0<y_1<L, y_2=1 \}.
\end{align}
Hence the nozzle $\mathscr{D}$ becomes a rectangle(see Figure \ref{fig:domainL})
\begin{equation}
\Omega= \{(y_1,y_2): 0<y_1<L, 0<y_2<1 \}.
\end{equation}
Similarly, the supersonic and subsonic region can be denoted as
\begin{align}
&\Omega_- = \{(y_1,y_2): 0<y_1<\psi_s(y_2) , 0<y_2<1 \}, \\
&\Omega_+ = \{(y_1,y_2): \psi_s(y_2)<y_1<L , 0<y_2<1 \}.
\end{align}
\begin{figure}[htbp]
	\centering
	\includegraphics[height=5cm, width=11cm]{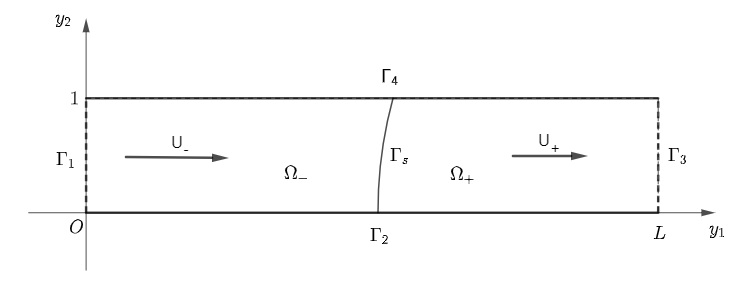}
	\caption{The domain under the Lagrange transformation}
	\label{fig:domainL}
\end{figure}

Thus the free boundary problem can be reformulated as follows:

\begin{problem} \label{problem: Lagrange}
Suppose that $\bar{U}_-$, $P_\mr{e}$ and $\Theta$ are given as Problem \ref{pr:SmallPerturbation}. Then look for a transonic shock solution $(U_-,U_+,\psi_s)$ such that
\begin{enumerate}
	\item $U_-$ and $U_+$ satisfy the reacting Euler system \eqref{LEuler1}-\eqref{LEuler5} in the domain $\Omega_-$ and $\Omega_+$ respectively.
	\item The state at the entrance coincides with the uniform supersonic state:
	\begin{equation}
	U_- = \bar{U}_-, \quad \text{on $\Gamma_1$},
	\end{equation}
	and the pressure at the exit is the given receiver pressure:
	\begin{equation}
	p_+ = \bar{p}_+ + P_\mr{e} (Y(L,y_2);\sigma,\kappa), \quad \text{on $\Gamma_3$},
	\end{equation}
	where
	\begin{equation}
	Y(L,y_2)=\int_0^{y_2} \frac{1}{(\rho q \cos \theta) (L,s)} ds.
	\end{equation}
	\item The slip boundary conditions hold along the boundary $W_1$ and $W_2$:
	\begin{align}
	&\theta_-(y_1) = 0,  &\text{on $\Gamma_2 \cap \overline{\Omega_-}$}, \\
	&\theta_-(y_1) = \sigma\Theta(y_1),  &\text{on $\Gamma_4 \cap \overline{\Omega_-}$}, \\
	&\theta_+(y_1) = 0,  &\text{on $\Gamma_2 \cap \overline{\Omega_+}$}, \\
	&\theta_+(y_1) = \sigma\Theta(y_1),  &\text{on $\Gamma_4 \cap \overline{\Omega_+}$}.
	\end{align}
	\item
	The R-H conditions \eqref{LRH1}-\eqref{LRH5} hold along the shock-front $\Gamma_s$.
\end{enumerate}
\end{problem}

By direct computations, \eqref{LEuler1}-\eqref{LEuler5} can be rewritten into non-divergence form as below:
\begin{align}
\label{NDF1}
&-\frac{\cos \theta}{\rho q} \frac{1-M^2}{\rho q^2} \partial_{y_1} p
-\frac{\sin \theta}{\rho q} \partial_{y_1}\theta + \partial_{y_2}\theta
= \kappa f_1(U), \\
\label{NDF2}
& -\frac{\sin \theta}{\rho q} \partial_{y_1} p + q\cos \theta \partial_{y_1}\theta + \partial_{y_2}p=0, \\
\label{NDF3}
& \partial_{y_1} p +\rho q \partial_{y_1}q =0, \\
\label{NDF4}
& \partial_{y_1} S = \kappa f_4(U),\\
\label{NDF5}
& \partial_{y_1} Z = -\kappa f_5(U),
\end{align}
where
\begin{align*}
	&f_1(U):=\frac{1}{\gamma c_v} \frac{1}{\rho q^2} \frac{\phi(T)}{T} {q_{\mr{e}}} Z, \\
	&f_4(U):=\frac{1}{q\cos \theta} \frac{\phi(T)}{T} {q_{\mr{e}}} Z, \\
	&f_5(U):=\frac{\phi(T)}{q \cos \theta} Z.
\end{align*}

Equation \eqref{NDF3} in the system can be replaced by
\begin{align}\label{NDF3'}
	\partial_{y_1} \left( \frac{1}{2} q^2 +\frac{\gamma p}{(\gamma-1)\rho}+ {q_{\mr{e}}} Z \right)=0.
\end{align}

It is obvious that \eqref{NDF3}-\eqref{NDF3'} are simple hyperbolic equations. To clarify the type of equations \eqref{NDF1}-\eqref{NDF2},  \eqref{NDF1}-\eqref{NDF2} can be rewritten into
\begin{equation}
A(U)
\partial_{y_1}
\left(\begin{matrix} p  \\ \theta \end{matrix}\right)
 +
 \partial_{y_2}
\left(\begin{matrix} p  \\ \theta \end{matrix}\right)
=
\left(\begin{matrix} 0 \\ \kappa f_1(U) \end{matrix}\right),
\end{equation}
where
\begin{equation}
A(U)=\frac{1}{\rho q}
\left(\begin{matrix} -\sin \theta & \rho q^2 \cos \theta
\\ 	-\cos \theta \cdot \displaystyle \frac{1-M^2}{\rho q^2} & -\sin \theta \end{matrix} \right).
\end{equation}
The eigenvalues of $A(U)$ are
\begin{equation}
\lambda_\pm =\frac{1}{\rho q} \left(-\sin \theta \pm \cos \theta \sqrt{M^2-1}\right),
\end{equation}
which implies \eqref{NDF1}-\eqref{NDF2} form a hyperbolic system when the flow is supersonic
and form an elliptic system when the flow is subsonic.
Hence equations \eqref{NDF1}-\eqref{NDF5} is a hyperbolic system in the supersonic region and an elliptic-hyperbolic composite system in the subsonic region.
Due to the theory of hyperbolic systems, The upstream flow can be solved in the whole nozzle if the perturbation of the boundary and the exothermic reaction is sufficiently small. However, the problem in the downstream region is a nonlinear elliptic-hyperbolic composite system and one of the boundaries, the shock front, is the very free boundary since there is no information of the location from the background solution.

To deal with the complicated free boundary problem, the key is to determine the location of the shock front. Next, we will start from a free boundary problem for the linearized reacting Euler system and get an approximating location of the shock front. And we will clarify how the boundaries, exothermic reaction and pressure at the exit influence the location of the shock front. Then the iteration scheme will be constructed based on the solution solved by the free boundary problem.

\subsection{The Free Boundary Problem for the Initial Approximation.}

Let $\overline{\psi}_s(y_2)\equiv\dot{\xi}$, with $\dot{\xi}\in\left(0,L\right)$ an unknown constant to be determined, and
\[
\dot{\Gamma}_s=\set{(y_1,y_2):\ y_1=\overline{\psi}_s(y_2),\ 0<y_2<1},
\]
which is taken as the initial approximating location of the shock-front. Clearly, $\overline{\psi}_s'(y_2)\equiv 0$. Then, $\dot{\Gamma}_s$ divides the domain $\Omega$ into two parts $\dot{\Omega}_{-}$ and $\dot{\Omega}_{+}$ as:
\begin{align*}
&\dot{\Omega}_{-} =  \set{(y_1,y_2):\ 0<y_1<\dot{\xi},\ 0<y_2<1},\\
&\dot{\Omega}_{+} =  \set{(y_1,y_2):\ \dot{\xi}<y_1<L,\ 0<y_2<1},
\end{align*}
and the boundaries $\Gamma_{2}$ and $\Gamma_{4}$  consist of
\begin{align*}
& \dot{\Gamma}_{2}^{-}  =  \set{(y_1,y_2):\ 0<y_1<\dot{\xi},\ y_2=0},\\
& \dot{\Gamma}_{2}^{+}  = \set{(y_1,y_2):\ \dot{\xi}<y_1<L,\ y_2=0},\\
& \dot{\Gamma}_{4}^{-}  =  \set{(y_1,y_2):\ 0<y_1<\dot{\xi},\ y_2=1},\\
& \dot{\Gamma}_{4}^{+}  =  \set{(y_1,y_2):\ \dot{\xi}<y_1<L,\ y_2=1}.
\end{align*}
\begin{figure}[htbp]
	\centering
	\includegraphics[height=5cm, width=11cm]{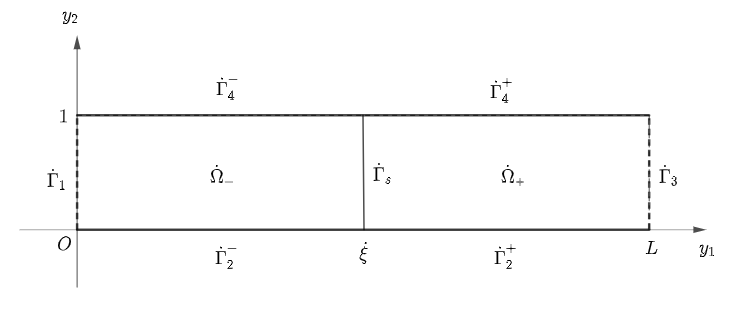}
	\caption{The domain for the linearized problem with an approximating shock-front}
	\label{fig:domain_initial}
\end{figure}

Assume that $\left(\bar{U}_{-},\bar{U}_{+}\right)$ is the unperturbed normal shock solution, with $\bar{U}_{-}=\left(\bar{p}_{-},0,\bar{q}_{-},\bar{S}_{-},\bar{Z}\right)^{\top}$ and $\bar{U}_{+}=\left(\bar{p}_{+},0,\bar{q}_{+},\bar{S}_{+},\bar{Z}\right)^{\top}$.
Let $\dot{U}_{-} \defs \left(\dot{p}_{-},\dot{\theta}_{-},\dot{q}_{-},\dot{S}_{-}, \dot{Z}_{-}\right)^{\top}$  satisfy the linearized reacting Euler system at the uniform supersonic state $\bar{U}_{-}$ below:
\begin{align}
& \partial_{y_2}\dot{p}_{-}+\overline{q}_{-}\partial_{y_1}\dot{\theta}_{-}=0,\label{eq:LEuler.1}\\
& \partial_{y_2}\dot{\theta}_{-}-\frac{1}{\overline{\rho}_{-}\overline{q}_{-}}\cdot\frac{1-\overline{M}_{-}^{2}}
{\overline{\rho}_{-}\overline{q}_{-}^{2}}\partial_{y_1}\dot{p}_{-}= \kappa \bar{f}_1^-,\label{eq:LEuler.2}\\
& \overline{\rho}_{-}\overline{q}_{-}\partial_{y_1}\dot{q}_{-}+\partial_{y_1}\dot{p}_{-}=0,\label{eq:LEuler.3}\\
& \partial_{y_1}\dot{S}_{-}
=\kappa \bar{f}_4^-,\label{eq:LEuler.4}\\
  & \partial_{y_1}\dot{Z}_{-}=-\kappa \bar{f}_5^-,\label{eq:LEuler.5}
\end{align}
where $\bar{f}_i^{-}:= f_i(\bar{U}_{-}), i=1,4,5$, i.e,
\begin{align*}
	&\bar{f}_1^-(U):=\frac{1}{\gamma c_{\mr{v}}}\cdot\frac{1}{\overline{\rho}_{-}\overline{q}_{-}^2}\cdot
	\frac{\phi(\overline{T}_-)}{\overline{T}_-}{q_{\mr{e}}}\bar{Z}, \\
	&\bar{f}_4^-(U):=\frac{1}{\overline{q}_-}\cdot\frac{\phi(\overline{T}_-)}{\overline{T}_-}{q_{\mr{e}}}\bar{Z}, \\
	&\bar{f}_5^-(U):=\frac{\phi(\overline{T}_-)}{\overline{q}_-}\bar{Z}.
\end{align*}

Let $\dot{U}_{+} \defs \left(\dot{p}_{+},\dot{\theta}_{+},\dot{q}_{+},\dot{S}_{+},\dot{Z}_{+}\right)^{\top}$ satisfy the linearized reacting Euler system at the uniform subsonic state $\overline{U}_{+}$ below:
\begin{align}
& \partial_{y_2}\dot{p}_{+}+\overline{q}_{+}\partial_{y_1}\dot{\theta}_{+}=0,\label{eq:LEuler.6}\\
& \partial_{y_2}\dot{\theta}_{+}-\frac{1}{\overline{\rho}_{+}\overline{q}_{+}}
\cdot\frac{1-\overline{M}_{+}^{2}}{\overline{\rho}_{+}\overline{q}_{+}^{2}}\partial_{y_1}\dot{p}_{+}
= \kappa \bar{f}_1^+,\label{eq:LEuler.7}\\
& \partial_{y_1}\left(\overline{q}_{+}\dot{q}_{+}+\frac{1}{\overline{\rho}_{+}}
\dot{p}_{+}+\overline{T}_{+}\dot{S}_{+}+{q_{\mr{e}}}\dot{Z}_{+}\right)=0,\label{eq:LEuler.8}\\
& \partial_{y_1}\dot{S}_{+}
= \kappa \bar{f}_4^+,\label{eq:LEuler.9}\\
  & \partial_{y_1}\dot{Z}_{+}=-\kappa \bar{f}_5^+,\label{eq:LEuler.10}
\end{align}
where $\bar{f}_i^{+}:= f_i(\bar{U}_{+}), i=1,4,5$, i.e,
\begin{align*}
	&\bar{f}_1^+(U):=\frac{1}{\gamma c_{\mr{v}}}\cdot\frac{1}{\overline{\rho}_{+}\overline{q}_{+}^2}\cdot
	\frac{\phi(\overline{T}_+)}{\overline{T}_+}{q_{\mr{e}}}\bar{Z}, \\
	&\bar{f}_4^+(U):=\frac{1}{\overline{q}_+}\cdot\frac{\phi(\overline{T}_+)}{\overline{T}_+}{q_{\mr{e}}}\bar{Z}, \\
	&\bar{f}_5^+(U):=\frac{\phi(\overline{T}_+)}{\overline{q}_+}\bar{Z}.
\end{align*}

Then the desired linear free boundary problem can be summarized as follows.
\begin{problem}  \label{problem: linearized}
Suppose that $\bar{U}_-$, $P_\mr{e}$ and $\Theta$ are given as Problem \ref{pr:SmallPerturbation}.
Determine an approximate solution $\left(\dot{U}_{-},\dot{U}_{+};\dot{\psi}_s', \dot{\xi}\right)$
such that
\begin{enumerate}
	\item $\dot{U}_{-}$ satisfies the linearized equations \eqref{eq:LEuler.1}-\eqref{eq:LEuler.5} in $\dot{\Omega}_{-}$ and the boundary conditions:
	\begin{align}
	& \dot{U}_{-}=0, \quad  \text{on} \ \  {\Gamma}_{1}, \label{eq:LBoun.1}\\
	& \dot{\theta}_{-}=0, \quad  \text{on} \ \  \dot{\Gamma}_{2}^{-}, \label{eq:LBoun.2}\\
	& \dot{\theta}_{-}=\sigma\Theta(y_1), \quad  \text{on} \ \  \dot{\Gamma}_{4}^{-}; \label{eq:LBoun.3}
	\end{align}
	\item $\dot{U}_{+}$ satisfies the linearized equations \eqref{eq:LEuler.6}-\eqref{eq:LEuler.10}	in $\dot{\Omega}_{+}$ and the boundary conditions:
	\begin{align}
	& \dot{p}_{+}=P_{\mr{e}}(y_2; \sigma, \kappa), \quad  \text{on} \ \  {\Gamma}_{3}, \label{eq:LBoun.4}\\
	& \dot{\theta}_{+}=0, \quad  \text{on} \ \  \dot{\Gamma}_{2}^{+},\label{eq:LBoun.5}\\
	& \dot{\theta}_{+}=\sigma\Theta(y_1), \quad  \text{on} \ \   \dot{\Gamma}_{4}^{+};\label{eq:LBoun.6}
	\end{align}
	\item Across the free boundary $\dot{\Gamma}_s$, where $\dot{\xi}\in(0,L)$ will be determined together with $\dot{U}_{-}$ and $\dot{U}_{+}$, $\left(\dot{U}_{-},\dot{U}_{+};\ \dot{\psi}_s' \right)$ satisfies the linearized R-H condition	at $\left(\overline{U}_{+},\overline{U}_{-};\ \overline{\psi}_s'\right)$ below:
	\begin{align}
	& \beta_{j}^{+}\cdot\dot{U}_{+}+\beta_{j}^{-}\cdot\dot{U}_{-}=0,\quad j=1,2,3,4,
\quad  \text{on} \ \    \dot{\Gamma}_{s},\label{eq:LRH.1}\\
	& \beta_{5}^{+}\cdot\dot{U}_{+}+\beta_{5}^{-}\cdot\dot{U}_{-}-\left[\overline{p}\right]\dot{\psi}_s'=0,
\quad  \text{on} \ \    \dot{\Gamma}_{s},\label{eq:LRH.2}
	\end{align}
	where the coefficients are given by $\beta_{j}^{\pm}=\nabla_{U_{\pm}}G_{j}|_{\left(\overline{U}_{+},\overline{U}_{-}\right)} , j=1,2,3,4,$  and $\beta_{5}^{\pm}=\nabla_{U_{\pm}}G_{5}|_{\left(\overline{U}_{+},\overline{U}_{-};\overline{\psi}_s'\right)}$.
\end{enumerate}
\end{problem}

By direct calculations, we have the following lemma.
\begin{lem}\label{lem:3.1}
The coefficients of the linearized R-H conditions \eqref{eq:LRH.1} and \eqref{eq:LRH.2} have explicit forms given below:
\begin{eqnarray*}
&&\beta_{1}^{\pm}=\nabla_{U_{\pm}}G_{1}|_{\left(\overline{U}_{+},\overline{U}_{-}\right)}  =  \pm\frac{1}{\overline{\rho}_{\pm}\overline{q}_{\pm}}\left[\overline{p}\right]
\cdot\left(-\frac{1}{\overline{\rho}_{\pm}\overline{c}_{\pm}^{2}},0,-\frac{1}{\overline{q}_{\pm}},\frac{1}{\gamma c_{\mr{v}}},0\right)^{\top},\\
&&\beta_{2}^{\pm}=\nabla_{U_{\pm}}G_{2}|_{\left(\overline{U}_{+},\overline{U}_{-}\right)}  =  \pm\frac{1}{\overline{\rho}_{\pm}\overline{q}_{\pm}}\left[\overline{p}\right]
\cdot\left(1-\frac{\overline{p}_{\pm}}{\overline{\rho}_{\pm}\overline{c}_{\pm}^{2}},0,
\overline{\rho}_{\pm}\overline{q}_{\pm}-\frac{\overline{p}_{\pm}}{\overline{q}_{\pm}},
\frac{\overline{p}_{\pm}}{\gamma c_{\mr{v}}},0\right)^{\top},\\
&&\beta_{3}^{\pm}=\nabla_{U_{\pm}}G_{3}|_{\left(\overline{U}_{+},\overline{U}_{-}\right)}  =
\pm\left(\frac{1}{\overline{\rho}_{\pm}},0,\overline{q}_{\pm},\frac{1}{\left(\gamma-1\right)c_{\mr{v}}}
\cdot\frac{\overline{p}_{\pm}}{\overline{\rho}_{\pm}},0\right)^{\top},\\
&&\beta_{4}^{\pm}=\nabla_{U_{\pm}}G_{4}|_{\left(\overline{U}_{+},\overline{U}_{-}\right)}  =  \pm\left(0,0,0,0,1\right)^{\top},\\
&&\beta_{5}^{\pm}=\nabla_{U_{\pm}}G_{5}|_{\left(\overline{U}_{+},\overline{U}_{-};\overline{\psi}_{\mr{s}}'\right)}
 =  \pm\left(0,\overline{q}_{\pm},0,0,0\right)^{\top}.
\end{eqnarray*}
\end{lem}

\subsection{Spaces and Notations}
Since the reacting Euler system is elliptic-hyperbolic composite in the subsonic region, we need to give different norms on different components. Denote $\bar{\psi}(y_2) \equiv \bar{\xi}_*$ and
\begin{align*}
	&\Omega(\bar{\psi}) = \{ (y_1,y_2) : \bar{\psi}(y_2)< y_1 < L, 0 < y_2 < 1\}, \\
	&\Gamma(\bar{\psi}) = \{ (y_1,y_2) : y_1 = \bar{\psi}(y_2), 0 < y_2 <1 \}.
\end{align*}
For the state $U=(p,\theta,q,S,Z)^{\top}$ in $\Omega(\bar{\psi})$, define
\begin{align*}
	\| U \|_{(\Omega(\bar{\psi}),\Gamma(\bar{\psi}))}
	&:= \| p \|_{W_\beta^1(\Omega(\bar{\psi}))} + \| \theta \|_{W_\beta^1(\Omega(\bar{\psi}))} \\
	&\quad + \| (q,S,Z) \|_{C(\overline{\Omega(\bar{\psi})})}
	+\| (q,S,Z) \|_{W_\beta^{1-\frac{1}{\beta}}(\Gamma(\bar{\psi}))}.
\end{align*}
For the subsonic region
\begin{align*}
	\Omega(\psi) = \{ (y_1,y_2) : \psi(y_2)< y_1 < L, 0 < y_2 < 1\},
\end{align*}
which has a shock front as a free boundary
\begin{align*}
	\Gamma(\psi) = \{ (y_1,y_2) : y_1 = \psi(y_2), 0 < y_2 <1 \},
\end{align*}
a transformation
$\Pi_{\psi}: \Omega(\psi) \rightarrow \Omega(\bar{\psi})$ needs to be introduced as follows:
\begin{align*}
	&z_1=L+\frac{L-\bar{\psi}(y_2)}{L-\psi(y_2)}(y_1-L) \\
	&z_2=y_2.
\end{align*}
The transformation is invertible if $\| \psi - \bar{\psi} \|_{C^{1,\alpha}[0,1]}$ is sufficiently small.
Then the norm of $U$ in $\Omega(\psi)$ can be defined as
\begin{align*}
	\| U \|_{(\Omega(\psi),\Gamma(\psi))}:=
	\| U \circ \Pi_\psi^{-1} \|_{(\Omega(\bar{\psi}),\Gamma(\bar{\psi}))}.
\end{align*}

\subsection{Main Theorems.} \label{section: main theorem}
The equations in the subsonic region is elliptic-hyperbolic composite. 
Its elliptic part is close to a Cauchy-Riemann system if we consider the solutions are near the background solution. 
With the boundary conditions of the subsonic region, additional conditions for $P_{\mr{e}}$ are necessary to make the elliptic sub-problem solvable (See \cite{FangXin2021CPAM}, Appendix A).
We will show the conditions first in the following and explain why we need these conditions in Section 3.

It turns out that the solvability condition will be 
\begin{align} \label{R=P}
	R(\dot{\xi}; \sigma, \kappa)= P_*(\sigma, \kappa).
\end{align}
where
\begin{align}
	R(\xi; \sigma, \kappa)
	&:=\sigma R_\sigma(\xi) +\kappa R_\kappa(\xi), \\
	P_*(\sigma, \kappa) &:=
	\frac{1}{\bar{\rho}_+ \bar{q}_+} \frac{1-\bar{M}_+^2}{\bar{\rho}_+ \bar{q}_+^2}
	\int_0^1 P_{\mr{e}} (y_2; \sigma, \kappa) \dif y_2,
\end{align}
with
\begin{align*}
	R_\sigma(\xi) &:= \int_0^L \Theta(y_1) \dif y_1-\dot{K}_1 \int_0^\xi \Theta(y_1) \dif y_1, \\
	R_\kappa(\xi) &:= -\bar{f}_1^+ L +\dot{K}_2 \xi,
\end{align*}
and
\begin{align*}
	\dot{K}_1 &:= [\bar{p}] \left( \frac{\gamma - 1}{\gamma \bar{p}_+}
	+ \frac{1}{\bar{\rho}_+ \bar{q}_+^2}\right) > 0, \\
	\dot{K}_2 &:=\frac{1}{\gamma c_{\mr{v}}} \frac{1}{\bar{T}_+} q_{\mr{e}} \bar{Z}
	(\frac{\phi(\overline{T}_+)}{\bar{\rho}_+ \bar{q}^2_+} - \frac{\phi(\overline{T}_-)}{\bar{\rho}_- \bar{q}^2_-})>0.
\end{align*}
Define
\begin{align*}
	\underline{R}(\sigma, \kappa) &:= \inf_{\xi \in (0,L)} R(\xi; \sigma, \kappa),\\
	\overline{R}(\sigma, \kappa) &:= \sup_{\xi \in (0,L)} R(\xi; \sigma, \kappa).
\end{align*}
For each $(\sigma, \kappa)$, if
\begin{align} \label{condition: admissible pressure}
	\underline{R}(\sigma, \kappa) <  P_*(\sigma, \kappa) < \overline{R}(\sigma, \kappa),
\end{align}
then by the continuity of the function $R$ with respect to $ \xi $, there exists at least one $\dot{\xi}(\sigma, \kappa)$ such that \eqref{R=P} holds.

In this paper, for explicitness of the argument, we assume that the perturbation of the pressure at the exit $P_\mr{e}$ has the following expression:
\begin{align*}
	P_\mr{e}(y_2; \sigma, \kappa) = \sigma P_\sigma(y_2) + \kappa P_\kappa(y_2),
\end{align*}
and set
\begin{align*}
	P_{\sigma *} &:=
	\frac{1}{\bar{\rho}_+ \bar{q}_+} \frac{1-\bar{M}_+^2}{\bar{\rho}_+ \bar{q}_+^2}
	\int_0^1 P_{\sigma}(y_2) \dif y_2, \\
	P_{\kappa *} &:=
	\frac{1}{\bar{\rho}_+ \bar{q}_+} \frac{1-\bar{M}_+^2}{\bar{\rho}_+ \bar{q}_+^2}
	\int_0^1 P_{\kappa}(y_2) \dif y_2.
\end{align*}
Next we will consider four typical cases. The following theorem shows that for these cases, there exists a fixed position $\bar{\xi}_* \in (0,L)$ such the approximating location of the shock front $\dot{\xi}(\sigma, \kappa)$ is close to $\bar{\xi}_*$ as $\sigma$ and $\kappa$ is sufficiently small.



\begin{thm} \label{thm: apprximating location of shock-front}
	Suppose $A,A_1,A_2$ are positive constants. $P_\mr{e}$ has the form
	\begin{align*}
		P_\mr{e}(y_2; \sigma, \kappa) = \sigma P_\sigma(y_2) + \kappa P_\kappa(y_2).
	\end{align*}
	Then
	\begin{itemize}
		\item[(i)]Suppose $\kappa = A_1 \sigma^s $ with $s > 1$ so that
		$P_\mr{e}(y_2;\sigma, \kappa) = \sigma P_\sigma(y_2) + \sigma^s A_1 P_\kappa(y_2)$.
		Denote
		\begin{align*}
			\overline{R_\sigma} = \sup_{\xi \in (0,L)} R_\sigma (\xi), \quad
			\underline{R_\sigma} = \inf_{\xi \in (0,L)} R_\sigma (\xi).
		\end{align*}
		If $P_{\sigma *} \in (\underline{R_\sigma}, \overline{R_\sigma})$,
		$\bar{\xi}_* \in (0,L)$ is a solution of $R_\sigma(\bar{\xi}_*) = P_{\sigma *}$ and
		$\Theta(\bar{\xi}_*) \neq 0$,
		then there exists a constant $\sigma_1 > 0$ such that
		for each $\sigma \in (0,\sigma_1)$, \eqref{R=P} has a solution $\dot{\xi}(\sigma, \kappa)$
		with the estimate
		\begin{align*}
			|\dot{\xi}(\sigma, \kappa) - \bar{\xi}_*| \leq C_2 \sigma^{s-1}.
		\end{align*}
		\item[(ii)]Suppose $\sigma = A_2 \kappa^s $ with $s > 1$ so that
		$P_\mr{e}(y_2;\sigma, \kappa) = \kappa P_\kappa(y_2) + \kappa^s A_2 P_\sigma(y_2)$.
		Denote 
		\begin{align*}
			\overline{R_\kappa} = \sup_{\xi \in (0,L)} R_\kappa (\xi), \quad
			\underline{R_\kappa} = \inf_{\xi \in (0,L)} R_\kappa (\xi).
		\end{align*}
		If $P_{\kappa *} \in (\underline{R_\kappa}, \overline{R_\kappa})$ and
		$\bar{\xi}_* \in (0,L)$ is a solution of $R_\kappa(\bar{\xi}_*) = P_{\kappa *}$,
		then there exists a constant $\kappa_1 > 0$ such that
		for each $\kappa \in (0,\kappa_1)$, \eqref{R=P} has a solution $\dot{\xi}(\sigma, \kappa)$
		with the estimate
		\begin{align*}
			|\dot{\xi}(\sigma, \kappa) - \bar{\xi}_*| \leq C_1 \kappa^{s-1}.
		\end{align*}
		\item[(iii)]Suppose $\sigma = A \kappa $ so that $P_\mr{e}(y_2;\sigma, \kappa) = \kappa P_A(y_2)$ with 
		\begin{align*}
			P_A(y_2)=A P_\sigma(y_2)+ P_\kappa(y_2).	
		\end{align*}
		Denote $R_A(\xi) = A R_\sigma(\xi) + R_\kappa(\xi)$ and
		\begin{align*}
			\overline{R_A} = \sup_{\xi \in (0,L)} R_A (\xi), \quad
			\underline{R_A} = \inf_{\xi \in (0,L)} R_A (\xi).
		\end{align*}
		If $P_{A*} \in (\underline{R_A}, \overline{R_A})$,
		where
		\begin{align*}
			P_{A*} :=
			\frac{1}{\bar{\rho}_+ \bar{q}_+} \frac{1-\bar{M}_+^2}{\bar{\rho}_+ 	\bar{q}_+^2}
			\int_0^1 P_{A}(y_2) \dif y_2,
		\end{align*}
		then there exists a constant $\bar{\xi}_* \in (0,L)$ such that \eqref{R=P} holds for 
		\begin{align*}
			\dot{\xi}(\sigma, \kappa) \equiv \bar{\xi}_*.
		\end{align*}
		\item[(iv)]
		Suppose $P_\sigma, P_\kappa$ satisfy $(P_{\sigma*}, P_{\kappa*}) \in \mathcal{C}$, where $\mathcal{C}$ is a curve which is defined by
		\begin{align*}
			\mathcal{C}:=\{ (y_1,y_2): (y_1,y_2)=(R_{\sigma}(\xi),R_{\kappa}(\xi)) , \xi \in (0,L)   \}.
		\end{align*} 
		Then there exists $\bar{\xi}_* \in (0,L)$ such that \eqref{R=P} holds for 		\begin{align*}
			\dot{\xi}(\sigma, \kappa) \equiv \bar{\xi}_*.
		\end{align*}
	\end{itemize}
\end{thm}

With the approximating location of the shock front solved by Theorem \ref{thm: apprximating location of shock-front}, the problem will be considered under four cases which will be called by \textbf{Perturbation Hypotheses}:
\begin{itemize}
	\item [(H1)] $\kappa = A_1 \sigma^s $ with $s > 1$ and $A_1>0$, 
	$P_{\sigma *} \in (\underline{R_\sigma}, \overline{R_\sigma})$ and $\Theta(\bar{\xi}_*) \neq 0$.
	\item [(H2)] $\sigma = A_2 \kappa^s $ with $s > 1$ and $A_2>0$, 
	and $P_{\kappa *} \in (\underline{R_\kappa}, \overline{R_\kappa})$.	
	\item [(H3)] $\sigma=A \kappa$ with $A>0$,
	$P_{A*} \in (\underline{R_A}, \overline{R_A})$ and $\Theta(\bar{\xi}_*) \neq \displaystyle \frac{\dot{K}_2}{A \dot{K}_1}$.
	\item [(H4)] $(P_{\sigma*}, P_{\kappa*}) \in \mathcal{C}$. $\sigma$ and $\kappa$ satisfy
	\begin{align} \label{admissble condition}
		| \kappa \dot{K}_2
		-\sigma \dot{K}_1 \Theta(\dot{\xi})| \geq \beta_0(\sigma + \kappa).
	\end{align}
	for some constant $\beta_0 >0$.
\end{itemize}

In Case (H1), the perturbation of the boundary of the nozzle has the main effect to the flow. The result is similar to the result in \cite{FangXin2021CPAM} where there is a perturbation on the boundary but no exothermic reaction. So the condition $\Theta(\bar{\xi}_*) \neq 0$ is also needed for this case. 
In Case (H2), the main effect is from the exothermic reaction. The behavior of the solution is similar to the case of a contracting nozzle, the location of the shock will get closer to the exit as the pressure increases at the exit. 
And for Case (H3), the influence of the exothermic reaction and the perturbation of the boundary of the nozzle is at the same level so that the shock front will be effected by both factors. So we consider the pressure at the exit and its range with a weight constant $A$.

Case (H4) is a special case that we add an additional condition on the pressure at the exit such that the approximating location of the shock front can be independent of $\sigma$ and $\kappa$. But for the estimates in the iteration scheme, we still need a technical condition \eqref{admissble condition} to give a lower bound. 

We note that it is easy to see that in Case (H1)-(H3) we can also choose a constant $\beta_0 >0$ such that \eqref{admissble condition} holds for sufficiently small $\sigma$ and $\kappa$. Therefore we also use the notation $\beta_0$ to represent the constant which makes the estimate \eqref{admissble condition} hold in Case (H1)-(H3).


Then for each case, the linearized problem can be solved as follows.

\begin{thm}\label{thm: linearized solution}
	Let $\alpha \in (0,1)$ and $\beta > 2$.
	Suppose that $P_\mr{e}$ and small constants $\sigma$, $\kappa$ match one of Perturbation Hypotheses.
	Then there exists a solution $\left( \dot{U}_-, \dot{U}_+; \dot{\psi}', \dot{\xi} \right)$ to the Problem \ref{problem: linearized}, where $\dot{\xi}$ is determined in the Theorem \ref{thm: apprximating location of shock-front}.
	Moreover, it holds that
	\begin{equation*}
		\norm{\dot{U}_{-}}_{\mcc^{2,\alpha}(\dot{\Omega}_{-})}
		+\norm{\dot{U}_{+}}_{(\dot{\Omega}_{+};\dot{\Gamma}_{s})}
		+\norm{\dot{\psi}'}_{W_{\beta}^{1-1/\beta}(\dot{\Gamma}_{s})}
		\leq \dot{C}(\sigma+\kappa),
	\end{equation*}
	where the constant $\dot{C}$ depends on $\overline{U}_{\pm}, L, \dot{\xi}$, $\alpha$ and $\beta$.
\end{thm}

With the approximating solution, now we can state the main theorem as follows.
\begin{thm} \label{Main Theorem}
	(Main Theorem)
	Let $\alpha \in (0,1)$ and $\beta > 2$.
	Suppose that $P_\mr{e}$ and small constants $\sigma$, $\kappa$ match one of Perturbation Hypotheses.
	$\dot{\xi}(\sigma, \kappa) \in (0,L)$ is the approximating location of the shock-front in Theorem \ref{thm: apprximating location of shock-front}.
	Then there exists a sufficiently small constant $\varepsilon>0$ depending on $\bar{U}_-$, $\bar{U}_+$, $L$, $\dot{\xi}$, $\alpha$, $\beta$ and $\beta_0$, such that for
	any $\sigma, \kappa$ satisfying $0<\sigma + \kappa < \varepsilon$,
	there exists a transonic shock solution $(U_-, U_+; \psi)$ to Problem \ref{problem: Lagrange}, with the estimates
	\begin{align*}
		&| \psi(1)-\dot{\xi} | \leq C_s (\sigma + \kappa),\\
		&\| \psi' \|_{W_{\beta}^{1-\frac{1}{\beta}}(\Gamma_s)} \leq C_s (\sigma + \kappa),\\
		&\| U_- - \bar{U}_- \|_{C^{2,\alpha}(\Omega_-)} \leq C_s (\sigma + \kappa),\\
		&\| U_+ - \bar{U}_+ \|_{(\Omega_+; \Gamma_s)} \leq C_s (\sigma + \kappa),
	\end{align*}
	where $C_s$ is a constant depending on $\bar{U}_-, \bar{U}_+, L$, $\dot{\xi}$, $\alpha$, $\beta$ and $\beta_0$.
	
	Furthermore, let $\left( \dot{U}_-, \dot{U}_+; \dot{\psi}', \dot{\xi} \right)$ be an approximating solution in Theorem \ref{thm: linearized solution}. Then the more accurate estimates hold:
	\begin{align*}
		&\| \psi' - \dot{\psi}' \|_{W_{\beta}^{1-\frac{1}{\beta}}(\Gamma_s)}
		\leq \frac{1}{2} (\sigma + \kappa)^{\frac{3}{2}},\\
		&\| U_- - (\bar{U}_- + \dot{U}_-) \|_{C^{1,\alpha}(\Omega_-)}
		\leq \frac{1}{2} (\sigma + \kappa)^{\frac{3}{2}},\\
		&\| U_+ \circ \Pi_{\psi}^{-1} - (\bar{U}_+ + \dot{U}_+) \|_{(\dot{\Omega}_+; \dot{\Gamma}_s)}
		\leq \frac{1}{2} (\sigma + \kappa)^{\frac{3}{2}}.
	\end{align*}
\end{thm}

\section{The Initial Approximation for the Shock Solution}
\subsection{The solution $\dot{U}_{-}$ in $\Omega$}
In the free boundary problem, it is easy to solve $\dot{U}_{-}$ in $\Omega$ since it satisfies a hyperbolic system with constant coefficients.
\begin{lem}\label{lem:3.2}
There exists a unique solution $\dot{U}_{-} \in C^{2,\alpha}(\Omega)$, which satisfies the linearized equations \eqref{eq:LEuler.1}-\eqref{eq:LEuler.5} in $\Omega$ with the boundary conditions \eqref{eq:LBoun.1}-\eqref{eq:LBoun.3}. Moreover, in the whole domain $\Omega$, it holds that
\begin{align}
 & \overline{\rho}_{-}\overline{q}_{-}\dot{q}_{-}+\dot{p}_{-}=0,\label{eq:q-}\\
 &\dot{S}_{-}(y_1,y_2)
 =\kappa\frac{1}{\overline{q}_-}\cdot\frac{\phi(\overline{T}_-)}{\overline{T}_-}{q_{\mr{e}}}\bar{Z}y_1,
 \label{eq:s-}\\
 &\dot{Z}_{-}(y_1,y_2)=-\kappa\frac{\phi(\overline{T}_-)}{\overline{q}_-}\bar{Z}y_1,\label{eq:z-}
\end{align}
and
\begin{eqnarray}
\norm{\dot{U}_{-}}_{\mcc^{2,\alpha}(\Omega)} & \leq & \dot{C}_{-}(\sigma+\kappa),\label{eq:U-}
\end{eqnarray}
where $\dot{C}_{-}$ is a constant depending on $\overline{U}_{-}$ and $L$.

Finally, for any $\xi\in\left(0,L\right)$, it holds that
\begin{eqnarray}
&&\int_{0}^{\xi}\sigma{\Theta}\left(y_1\right)\dif y_1+
\frac{1}{\overline{\rho}_{-}\overline{q}_{-}}\cdot\frac{\overline{M}_{-}^{2}-1}{\overline{\rho}_{-}\overline{q}_{-}^{2}}
\cdot\int_{0}^{1}\dot{p}_{-}\left(\xi,y_2\right)\dif y_2\nonumber\\
&=&\kappa \frac{1}{\gamma c_{\mr{v}}}\cdot\frac{1}{\overline{\rho}_{-}\overline{q}_{-}^2}\cdot
\frac{\phi(\overline{T}_-)}{\overline{T}_-}{q_{\mr{e}}}\bar{Z} \xi.
\label{eq:p-}
\end{eqnarray}
\end{lem}
\begin{proof}
By \eqref{eq:LEuler.1}, there exists a function $\phi_{-}$, such that
\[
\partial_{y_1}\phi_{-}=-\dot{p}_{-},\qquad
\partial_{y_2}\phi_{-}=\overline{q}_{-}\dot{\theta}_{-}.
\]
Then substitute into \eqref{eq:LEuler.2} to get
\[
\frac{1}{\overline{q}_{-}}\partial_{y_2}^2\phi_{-}
-\frac{1}{\overline{\rho}_{-}\overline{q}_{-}}\cdot\frac{\overline{M}_{-}^{2}-1}
{\overline{\rho}_{-}\overline{q}_{-}^{2}}\partial_{y_1}^2\phi_{-}=
\kappa \frac{1}{\gamma c_{\mr{v}}}\cdot\frac{1}{\overline{\rho}_{-}\overline{q}_{-}^2}\cdot
\frac{\phi(\overline{T}_-)}{\overline{T}_-}{q_{\mr{e}}}\bar{Z},
\]
and the boundary conditions become
\begin{align*}
&\partial_{y_1}\phi_{-}=0, \quad  \text{on} \ \  {\Gamma}_{1}, \\
	& \partial_{y_2}\phi_{-}=0, \quad  \text{on} \ \  {\Gamma}_{2}, \\
	&\partial_{y_2}\phi_{-}=\sigma\overline{q}_{-}\Theta(y_1), \quad  \text{on} \ \  {\Gamma}_{4}.
\end{align*}

For any $\xi\in\left(0,L\right)$, integrating in the domain $[0,\xi]\times[0,1]$, it holds that
\begin{eqnarray*}
&&\int_{0}^{\xi}\int_{0}^{1}\kappa \frac{1}{\gamma c_{\mr{v}}}\cdot\frac{1}{\overline{\rho}_{-}\overline{q}_{-}^2}\cdot
\frac{\phi(\overline{T}_-)}{\overline{T}_-}{q_{\mr{e}}}\bar{Z}\dif y_2\dif y_1\\
&=&\int_{0}^{\xi}\int_{0}^{1}\frac{1}{\overline{q}_{-}}\partial_{y_2}^2\phi_{-}\dif y_2\dif y_1
-\int_{0}^{1}\int_{0}^{\xi}\frac{1}{\overline{\rho}_{-}\overline{q}_{-}}\cdot\frac{\overline{M}_{-}^{2}-1}
{\overline{\rho}_{-}\overline{q}_{-}^{2}}\partial_{y_1}^2\phi_{-}\dif y_1\dif y_2\\
&=&\int_{0}^{\xi}\sigma{\Theta}\left(y_1\right)\dif y_1
+\frac{1}{\overline{\rho}_{-}\overline{q}_{-}}\cdot
\frac{\overline{M}_{-}^{2}-1}{\overline{\rho}_{-}\overline{q}_{-}^{2}}\int_{0}^{1} \dot{p}_{-}(\xi,y_2)\dif y_2,
\end{eqnarray*}
which is the condition \eqref{eq:p-} exactly.
\end{proof}

\subsection{Reformulation of the linearized boundary conditions \eqref{eq:LRH.1}}
With $\dot{U}_{-}$ being determined, one then need to determine $\dot{U}_{+}$ and $\dot{\xi}$. The linearized boundary conditions \eqref{eq:LRH.1} can be rewritten as
\begin{equation}
B_{s}\cdot\begin{bmatrix}\dot{p}_{+}\\
\dot{q}_{+}\\
\dot{S}_{+}\\
\dot{Z}_{+}
\end{bmatrix}=\begin{bmatrix}\dot{g}_{1}\\
\dot{g}_{2}\\
\dot{g}_{3}\\
\dot{g}_{4}
\end{bmatrix},\label{eq:qps+}
\end{equation}
where $\dot{g}_{j}\defs-\beta_{j}^{-}\cdot\dot{U}_{-}$, $(j=1,2,3)$,
and the coefficient matrix is given by
\[
B_{s}\defs\frac{1}{\overline{\rho}_{+}\overline{q}_{+}}\left[\overline{p}\right]\cdot
\begin{bmatrix}
-\displaystyle\frac{1}{\overline{\rho}_{+}\overline{c}_{+}^{2}} &-\displaystyle\frac{1}{\overline{q}_{+}}& \displaystyle\frac{1}{\gamma c_{\mr{v}}} & 0\\
1-\displaystyle\frac{\overline{p}_{+}}{\overline{\rho}_{+}\overline{c}_{+}^{2}}
&\overline{\rho}_{+}\overline{q}_{+}-\displaystyle\frac{\overline{p}_{+}}{\overline{q}_{+}} & \displaystyle\frac{\overline{p}_{+}}{\gamma c_{\mr{v}}} & 0\\
\displaystyle\frac{\overline{q}_{+}}{\left[\overline{p}\right]}
&\displaystyle\frac{\overline{\rho}_{+}\overline{q}_{+}^{2}}{\left[\overline{p}\right]} & \displaystyle\frac{1}{\left(\gamma-1\right)c_{\mr{v}}}\cdot\overline{p}_{+}\cdot\frac{\overline{q}_{+}}{\left[\overline{p}\right]} & 0 \\
0 & 0 & 0 & \displaystyle \frac{\overline{\rho}_{+}\overline{q}_{+}}{\left[\overline{p}\right]}
\end{bmatrix}.
\]
\begin{lem}
It follows from the boundary conditions \eqref{eq:qps+} that
\begin{eqnarray}
\det B_{s} & = & \frac{1}{\left(\gamma-1\right)c_{\mr{v}}}\cdot\frac{\left[\overline{p}\right]^{2}
\cdot\overline{p}_{+}}{\left(\overline{\rho}_{+}\overline{q}_{+}\right)^{3}}
\cdot\left(1-\overline{M}_{+}^{2}\right)\neq 0,\label{eq:Bs}\\
\dot{p}_{+} & = & \dot{g}_{1}^{\#}\defs\frac{\overline{\rho}_{+}\overline{q}_{+}^{2}}{\overline{M}_{+}^{2}-1}
\cdot\frac{\overline{M}_{-}^{2}-1}{\overline{\rho}_{-}\overline{q}_{-}^{2}}
\cdot\left(1-\dot{K}_1\right)\cdot\dot{p}_{-}\nonumber\\
&&\hspace{3em}+\frac{1}{\gamma c_{\mr{v}}}\cdot\frac{\overline{\rho}_{+}\overline{q}_{+}^{2}}{\overline{M}_{+}^{2}-1}
\cdot\left(
\frac{\overline{T}_{-}}{\overline{T}_{+}}-1+\dot{K}_1\right)\cdot\dot{S}_{-},\label{eq:p+}\\
\dot{q}_{+} & = & \dot{g}_{2}^{\#}\defs\frac{\overline{M}_{-}^{2}-1}{\overline{\rho}_{-}\overline{q}_{-}^{2}}\cdot\left\{ \left[\overline{p}\right]-\frac{\overline{\rho}_{+}\overline{q}_{+}^{2}}{\overline{M}_{+}^{2}-1}
\left(1-\dot{K}_1\right)\right\}\cdot\dot{p}_{-}\nonumber\\
&&\hspace{3em}-\frac{1}{\gamma c_{\mr{v}}}\cdot\left\{
\left[\overline{p}\right]+\frac{\overline{\rho}_{+}\overline{q}_{+}^{2}}{\overline{M}_{+}^{2}-1}
\left(\frac{\overline{T}_{-}}{\overline{T}_{+}}-1+\dot{K}_1\right)\right\}\cdot\dot{S}_{-},\label{eq:q+}\\
\dot{S}_{+} & = & \dot{g}_{3}^{\#}\defs-\left(\gamma-1\right)c_{\mr{v}}
\cdot\frac{\overline{M}_{-}^{2}-1}{\overline{\rho}_{-}\overline{q}_{-}^{2}}
\cdot\frac{\left[\overline{p}\right]}{\overline{p}_{+}}\cdot\dot{p}_{-}\nonumber\\
&&\hspace{3em}+\left(\frac{\overline{T}_{-}}{\overline{T}_{+}}+\frac{\gamma-1}{\gamma}\cdot
\frac{\left[\overline{p}\right]}{\overline{p}_{+}}\right)\cdot\dot{S}_{-},\label{eq:s+}\\
\dot{Z}_{+} & = & \dot{g}_{4}^{\#}\defs \dot{Z}_-,
\end{eqnarray}
where $\dot{K}_1\defs\left[\overline{p}\right]\left(\displaystyle\frac{\gamma-1}{\gamma\overline{p}_{+}}
+\frac{1}{\overline{\rho}_{+}\overline{q}_{+}^{2}}\right)>0$.
\end{lem}
\begin{proof}
The boundary conditions \eqref{eq:qps+} for $j=1$ can be written as
\begin{eqnarray*}
 && \frac{1}{\overline{\rho}_{+}\overline{q}_{+}^{2}}\dot{q}_{+}+
 \frac{1}{\overline{\rho}_{+}^{2}\overline{c}_{+}^{2}\overline{q}_{+}}\dot{p}_{+}-\frac{1}{\gamma c_{\mr{v}}}\cdot\frac{1}{\overline{\rho}_{+}\overline{q}_{+}}\dot{S}_{+}\\
&= & \frac{1}{\overline{\rho}_{-}\overline{q}_{-}^{2}}\dot{q}_{-}+
\frac{1}{\overline{\rho}_{-}^{2}\overline{c}_{-}^{2}\overline{q}_{-}}\dot{p}_{-}-\frac{1}{\gamma c_{\mr{v}}}\cdot\frac{1}{\overline{\rho}_{-}\overline{q}_{-}}\dot{S}_{-},
\end{eqnarray*}
which yield, by  $\overline{\rho}_{+}\overline{q}_{+}=\overline{\rho}_{-}\overline{q}_{-}$ and \eqref{eq:q-}, that
\begin{eqnarray}
\frac{1}{\overline{q}_{+}}\dot{q}_{+}+\frac{1}{\overline{\rho}_{+}\overline{c}_{+}^{2}}\dot{p}_{+}-\frac{1}{\gamma c_{\mr{v}}}\dot{S}_{+} & = & \frac{1}{\overline{q}_{-}}\dot{q}_{-}+\frac{1}{\overline{\rho}_{-}\overline{c}_{-}^{2}}\dot{p}_{-}
-\frac{1}{\gamma c_{\mr{v}}}\dot{S}_{-}\nonumber \\
 & = & \frac{1}{\overline{\rho}_{-}\overline{q}_{-}^{2}}\left(\overline{M}_{-}^{2}-1\right)\dot{p}_{-}
 -\frac{1}{\gamma c_{\mr{v}}}\dot{S}_{-}.\label{eq:2.41.1}
\end{eqnarray}

Similarly, the boundary condition \eqref{eq:qps+} for $j=2$ leads to
\begin{eqnarray*}
 && \left(\overline{\rho}_{+}\overline{q}_{+}\dot{q}_{+}+\dot{p}_{+}\right)-\overline{p}_{+}
 \left(\frac{1}{\overline{q}_{+}}
 \dot{q}_{+}+\frac{1}{\overline{\rho}_{+}\overline{c}_{+}^{2}}\dot{p}_{+}-\frac{1}{\gamma c_{\mr{v}}}\dot{S}_{+}\right)\\
&= & \left(\overline{\rho}_{-}\overline{q}_{-}\dot{q}_{-}+\dot{p}_{-}\right)-\overline{p}_{-}
\left(\frac{1}{\overline{q}_{-}}\dot{q}_{-}+\frac{1}{\overline{\rho}_{-}\overline{c}_{-}^{2}}\dot{p}_{-}
-\frac{1}{\gamma c_{\mr{v}}}\dot{S}_{-}\right).
\end{eqnarray*}
Then, combining \eqref{eq:2.41.1} with \eqref{eq:q-} yields
\begin{equation}
\overline{\rho}_{+}\overline{q}_{+}\dot{q}_{+}+\dot{p}_{+}
=\frac{\left[\overline{p}\right]}{\overline{\rho}_{-}\overline{q}_{-}^{2}}\left(\overline{M}_{-}^{2}-1\right)
\dot{p}_{-}-\frac{\left[\overline{p}\right]}{\gamma c_{\mr{v}}}\dot{S}_{-}.\label{eq:2.41.2}
\end{equation}

The boundary condition \eqref{eq:qps+} for $j=3$ reads
\begin{eqnarray*}
 && \overline{q}_{+}\dot{q}_{+}+ \frac{1}{\overline{\rho}_{+}}\dot{p}_{+}
 +\frac{1}{\left(\gamma-1\right)c_{\mr{v}}}
 \cdot\frac{\overline{p}_{+}}{\overline{\rho}_{+}}\dot{S}_{+}\\
&= & \overline{q}_{-}\dot{q}_{-}+ \frac{1}{\overline{\rho}_{-}}\dot{p}_{-}
+\frac{1}{\left(\gamma-1\right)c_{\mr{v}}}
\cdot\frac{\overline{p}_{-}}{\overline{\rho}_{-}}\dot{S}_{-},
\end{eqnarray*}
which yields, by \eqref{eq:q-} again, that
\[
\frac{1}{\overline{\rho}_{+}}\left(\overline{\rho}_{+}\overline{q}_{+}\dot{q}_{+}+\dot{p}_{+}\right)
+\frac{1}{\left(\gamma-1\right)c_{\mr{v}}}\cdot\frac{\overline{p}_{+}}{\overline{\rho}_{+}}\dot{S}_{+}
=\frac{1}{\left(\gamma-1\right)c_{\mr{v}}}
\cdot\frac{\overline{p}_{-}}{\overline{\rho}_{-}}\dot{S}_{-}.
\]
Hence, this and \eqref{eq:2.41.2} yield that
\begin{equation}
\frac{1}{c_{\mr{v}}}\cdot\dot{S}_{+}=-\left(\gamma-1\right)\cdot\frac{\overline{M}_{-}^{2}-1}
{\overline{\rho}_{-}\overline{q}_{-}^{2}}\cdot\frac{\left[\overline{p}\right]}
{\overline{p}_{+}}\dot{p}_{-}+\frac{1}{c_{\mr{v}}}\left(\frac{\overline{\rho}_{+}}{\overline{p}_{+}}\cdot
\frac{\overline{p}_{-}}{\overline{\rho}_{-}}+\frac{\gamma-1}{\gamma}\cdot
\frac{\left[\overline{p}\right]}{\overline{p}_{+}}\right)\dot{S}_{-},
\label{eq:2.41.3}
\end{equation}
which is \eqref{eq:s+} exactly since $\overline{T}_{\pm} =\displaystyle \frac{1}{\mathcal{R}}  \frac{\overline{p}_{\pm}} {\overline{\rho}_{\pm}}$ for the constant $\mathcal{R}$.

Then, substituting \eqref{eq:2.41.3} into \eqref{eq:2.41.1} gives
\begin{eqnarray*}
&&\frac{1}{\overline{\rho}_{+}\overline{q}_{+}^{2}}\left(\overline{\rho}_{+}\overline{q}_{+}\dot{q}_{+}
+\overline{M}_{+}^{2}\dot{p}_{+}\right)\\
&=&\frac{1}{\overline{q}_{+}}\dot{q}_{+}+\frac{1}{\overline{\rho}_{+}\overline{c}_{+}^{2}}\dot{p}_{+}\\
& = & \frac{1}{\gamma c_{\mr{v}}}\dot{S}_{+}+\frac{1}{\overline{\rho}_{-}\overline{q}_{-}^{2}}
\left(\overline{M}_{-}^{2}-1\right)\dot{p}_{-}-\frac{1}{\gamma c_{\mr{v}}}\dot{S}_{-}\\
 & = & \left(1-\frac{\gamma-1}{\gamma}\cdot\frac{\left[\overline{p}\right]}{\overline{p}_{+}}\right)
 \cdot\frac{\overline{M}_{-}^{2}-1}{\overline{\rho}_{-}\overline{q}_{-}^{2}}\dot{p}_{-}
 +\frac{1}{\gamma c_{\mr{v}}}\left(\frac{\overline{\rho}_{+}}{\overline{p}_{+}}\cdot
\frac{\overline{p}_{-}}{\overline{\rho}_{-}}+\frac{\gamma-1}{\gamma}\cdot
\frac{\left[\overline{p}\right]}{\overline{p}_{+}}-1\right)\dot{S}_{-}.
\end{eqnarray*}
This, together with \eqref{eq:2.41.2}, gives
\begin{eqnarray}
\frac{\overline{M}_{+}^{2}-1}{\overline{\rho}_{+}\overline{q}_{+}^{2}}\cdot\dot{p}_{+}
&=&\frac{\overline{M}_{-}^{2}-1}{\overline{\rho}_{-}\overline{q}_{-}^{2}}\cdot\dot{p}_{-}\left\{ 1-\left[\overline{p}\right]\left(\frac{\gamma-1}{\gamma\overline{p}_{+}}+\frac{1}
{\overline{\rho}_{+}\overline{q}_{+}^{2}}\right)\right\}\nonumber\\
&&+\frac{1}{\gamma c_{\mr{v}}}\cdot\dot{S}_{-}\left\{\frac{\overline{\rho}_{+}}{\overline{p}_{+}}\cdot
\frac{\overline{p}_{-}}{\overline{\rho}_{-}}-1+\left[\overline{p}\right]
\left(\frac{\gamma-1}{\gamma\overline{p}_{+}}+\frac{1}
{\overline{\rho}_{+}\overline{q}_{+}^{2}}\right)\right\},\label{eq:2.41.4}
\end{eqnarray}
which yields \eqref{eq:p+}.

Finally, substituting \eqref{eq:2.41.4} into \eqref{eq:2.41.2} shows \eqref{eq:q+}. This completes the proof.
\end{proof}

\subsection{Determine $\dot{U}_{+}$ and $\dot{\xi}$ }
Now we determine $\dot{U}_{+}$ and $\dot{\xi}$. The key
step is determining $\dot{\xi}$ and $(\dot{p}_{+},\dot{\theta}_{+})$
via the elliptic system of first order consisting in the equations
\eqref{eq:LEuler.6}-\eqref{eq:LEuler.7} together with the boundary conditions \eqref{eq:LBoun.4}-\eqref{eq:LBoun.6}, and \eqref{eq:p+}. Then we have the following lemma.
\begin{lem}
Given $\dot{\xi} \in(0,L)$, there exists a unique solution $(\dot{p}_{+},\dot{\theta}_{+})$
to the boundary value problem consisting of the equations \eqref{eq:LEuler.6}-\eqref{eq:LEuler.7}
together with the boundary conditions \eqref{eq:LBoun.4}-\eqref{eq:LBoun.6}, and \eqref{eq:p+}, if and only if
\begin{eqnarray}
&&\frac{1}{\overline{\rho}_{+}\overline{q}_{+}}\cdot\frac{1-\overline{M}_{+}^{2}}{\overline{\rho}_{+}
\overline{q}_{+}^{2}}\cdot\int_{0}^{1}
 P_{\mr{e}}(y_2; \sigma, \kappa)  \dif y_2\nonumber\\
&=&\sigma \left(\int_{0}^{L}{\Theta}\left(y_1\right)\dif y_1
-\dot{K}_1 \int_{0}^{\dot{\xi}}{\Theta}\left(y_1\right)\dif y_1 \right)
+\kappa \left( -\bar{f}_1^+ L + \dot{K}_2 \dot{\xi} \right),
\label{eq:2.43}
\end{eqnarray}
where
\begin{align*}
\dot{K}_1 &\defs \left[\overline{p}\right]\left(\frac{\gamma-1}{\gamma\overline{p}_{+}}
+\frac{1}{\overline{\rho}_{+}\overline{q}_{+}^{2}}\right)>0, \\
\dot{K}_2 &\defs \frac{1}{\gamma c_{\mr{v}}} \frac{1}{\bar{T}_+} q_{\mr{e}} \bar{Z}
(\frac{\phi(\overline{T}_+)}{\bar{\rho}_+ \bar{q}^2_+} - \frac{\phi(\overline{T}_-)}{\bar{\rho}_- \bar{q}^2_-})>0.
\end{align*}

Furthermore, if $P_\mr{e}$ has the form $P_\mr{e}(y_2; \sigma, \kappa)=\sigma P_\sigma(y_2) + \kappa P_\kappa(y_2)$ for $\beta>2$, the solution $(\dot{p}_{+},\dot{\theta}_{+})$ satisfies the following estimate:
\begin{align}\label{eq:ptheta+}
 \norm{\dot{p}_{+}}_{W_{\beta}^{1}(\dot{\Omega}_{+})}
+\norm{\dot{\theta}_{+}}_{W_{\beta}^{1}(\dot{\Omega}_{+})}
\leq \dot{C}_{+}(\sigma+\kappa),
\end{align}
where the constant $\dot{C}_{+}$ depends on $\overline{U}_{+}, L, \dot{\xi}$ and $\alpha$.
\end{lem}
\begin{proof}
By the solvability condition for the elliptic system \eqref{eq:LEuler.6}-\eqref{eq:LEuler.7} with the boundary conditions \eqref{eq:LBoun.4}-\eqref{eq:LBoun.6}, and \eqref{eq:p+}, it holds that
\begin{eqnarray}
&&\kappa\frac{1}{\gamma c_{\mr{v}}}\cdot\frac{1}{\overline{\rho}_{+}\overline{q}_{+}^2}\cdot
\frac{\phi(\overline{T}_+)}{\overline{T}_+}{q_{\mr{e}}}\bar{Z}(L-\dot{\xi})\nonumber\\
&=&\frac{1}{\overline{\rho}_{+}\overline{q}_{+}}\cdot\frac{1-\overline{M}_{+}^{2}}
{\overline{\rho}_{+}\overline{q}_{+}^{2}}\int_{0}^{1}\left(\dot{p}_{+}
(\dot{\xi},y_2)- P_{\mr{e}}(y_2;\sigma, \kappa)\right)\dif y_2\nonumber\\
&&+\int_{\dot{\xi}}^{L}
\sigma{\Theta}\left(y_1\right)\dif y_1.\label{eq:2.44}
\end{eqnarray}

Then by \eqref{eq:s-}, \eqref{eq:p-}, \eqref{eq:p+}, and recalling that  $\overline{\rho}_{+}\overline{q}_{+}=\overline{\rho}_{-}\overline{q}_{-}$, one gets that
\begin{eqnarray*}
&&\frac{1}{\overline{\rho}_{+}\overline{q}_{+}}\cdot\frac{1-\overline{M}_{+}^{2}}{\overline{\rho}_{+}
\overline{q}_{+}^{2}}\int_{0}^{1}\dot{p}_{+}(\dot{\xi},y_2)\dif y_2 \\
& = & \frac{1}{\overline{\rho}_{-}\overline{q}_{-}}\cdot\frac{1-\overline{M}_{-}^{2}}{\overline{\rho}_{-}
\overline{q}_{-}^{2}}\cdot\left(1-\dot{K}_1\right)\int_{0}^{1}\dot{p}_{-}(\dot{\xi},y_2)\dif y_2\\
&&-\frac{1}{\gamma c_{\mr{v}}}\cdot\frac{1}{\overline{\rho}_{-}\overline{q}_{-}}
\cdot\left(\frac{\overline{T}_{-}}{\overline{T}_{+}}-1+\dot{K}_1\right)\int_{0}^{1}\dot{S}_{-}(\dot{\xi}, y_2)\dif y_2\\
 & = &  \left(1-\dot{K}_1\right)
 \int_{0}^{\dot{\xi}}\sigma{\Theta}\left(y_1\right)\dif y_1
-\kappa\frac{1}{\gamma c_{\mr{v}}}\cdot\frac{1}{\overline{\rho}_{-}\overline{q}_{-}^2}\cdot
\frac{\phi(\overline{T}_-)}{\overline{T}_-}
\left(1-\dot{K}_1\right){q_{\mr{e}}}\bar{Z}\dot{\xi}\\
 &&-\kappa\frac{1}{\gamma c_{\mr{v}}}\cdot\frac{1}{\overline{\rho}_{-}\overline{q}_{-}^2}
\cdot\frac{\phi(\overline{T}_-)}{\overline{T}_-}
\left(\frac{\overline{T}_{-}}{\overline{T}_{+}}-1+\dot{K}_1\right){q_{\mr{e}}}\bar{Z}\dot{\xi}\\
&=& \left(1-\dot{K}_1 \right)
 \int_{0}^{\dot{\xi}}\sigma{\Theta}\left(y_1\right)\dif y_1
 -\kappa\frac{1}{\gamma c_{\mr{v}}}\cdot\frac{1}{\overline{\rho}_{-}\overline{q}_{-}^2}
\cdot\frac{\overline{T}_{-}}{\overline{T}_{+}}\cdot \frac{\phi(\overline{T}_-)}{\overline{T}_-}{q_{\mr{e}}}\bar{Z}\dot{\xi}.
\end{eqnarray*}
Therefore, \eqref{eq:2.44} yields that
\begin{eqnarray*}
&&
\frac{1}{\overline{\rho}_{+}\overline{q}_{+}}\cdot\frac{1-\overline{M}_{+}^{2}}
{\overline{\rho}_{+}\overline{q}_{+}^{2}}\int_{0}^{1}
P_{\mr{e}}(y_2; \sigma, \kappa) \dif y_2\\
&=&\left(1-\dot{K}_1\right)
 \int_{0}^{\dot{\xi}}\sigma{\Theta}\left(y_1\right)\dif y_1+\int_{\dot{\xi}}^{L}
\sigma{\Theta}\left(y_1\right)\dif y_1\\
&&\hspace{1em}-\kappa\frac{1}{\gamma c_{\mr{v}}}\cdot\frac{1}{\overline{\rho}_{-}\overline{q}_{-}^2}
\cdot\frac{\overline{T}_{-}}{\overline{T}_{+}}\cdot \frac{\phi(\overline{T}_-)}{\overline{T}_-}{q_{\mr{e}}}\bar{Z}\dot{\xi}-\kappa\frac{1}{\gamma c_{\mr{v}}}\cdot\frac{1}{\overline{\rho}_{+}\overline{q}_{+}^2}\cdot
\frac{\phi(\overline{T}_+)}{\overline{T}_+}{q_{\mr{e}}}\bar{Z}(L-\dot{\xi}) \\
&=& \sigma \left(\int_{0}^{L}{\Theta}\left(y_1\right)\dif y_1
-\dot{K}_1 \int_{0}^{\dot{\xi}}{\Theta}\left(y_1\right)\dif y_1 \right) \\
&& \hspace{1em} -\kappa \frac{1}{\gamma c_{\mr{v}}}\cdot\frac{1}{\overline{\rho}_{+}\overline{q}_{+}^2}\cdot
\frac{\phi(\overline{T}_+)}{\overline{T}_+}{q_{\mr{e}}}\bar{Z} L  +\kappa \dot{\xi} \frac{1}{\gamma c_{\mr{v}}}\cdot\frac{1}{\overline{T}_{+}} \cdot q_{\mr{e}}\bar{Z}
\left( \frac{\phi(\overline{T}_+)}{\overline{\rho}_{+}\overline{q}_+^2}
-\frac{\phi(\overline{T}_-)}{\overline{\rho}_{-}\overline{q}_-^2}  \right),
\end{eqnarray*}
which is the condition \eqref{eq:2.43} exactly.

Moreover, if $P_\mr{e}$ has the form $P_\mr{e}(y_2; \sigma, \kappa)=\sigma P_\sigma(y_2) + \kappa P_\kappa(y_2)$ for $\beta>2$, it holds that
\begin{eqnarray*}
&&\norm{\dot{p}_{+}}_{W_{\beta}^{1}(\dot{\Omega}_{+})}
+\norm{\dot{\theta}_{+}}_{W_{\beta}^{1}(\dot{\Omega}_{+})}\\
&\le& \dot{C}_{+}\left\{\sigma \norm{{\Theta}}_{\mcc^{2,\alpha}(\Gamma_{4})}
+\sigma\norm{ P_{\sigma}}_{\mcc^{2,\alpha}(\Gamma_{3})}
+ \kappa\norm{ P_{\kappa}}_{\mcc^{2,\alpha}(\Gamma_{3})}
+\norm{\dot{g}_{1}^{\#}}_{W_{\beta}^{1-1/\beta}(\dot{\Gamma}_{s})}\right\}\\
&\leq& \dot{C}_{+}(\sigma+\kappa),
\end{eqnarray*}
where the constant $\dot{C}_{+}$ depends on $\overline{U}_{+}, L, \dot{\xi}$ and $\alpha$.
\end{proof}

Once $\dot{\xi}$ is determined, then we can determine $\dot{U}_+$ in the domain $\dot{\Omega}_+$.
It follows from \eqref{eq:LEuler.8}-\eqref{eq:LEuler.10} that
\begin{align*}
 & \left(\overline{q}_{+}\dot{q}_{+}+\frac{1}{\overline{\rho}_{+}}\dot{p}_{+}+\overline{T}_{+}\dot{S}_{+}
 +{q_{\mr{e}}}\dot{Z}_{+}\right)\Big|_{(y_1,y_2)}
 =\left(\overline{q}_{+}\dot{q}_{+}+\frac{1}{\overline{\rho}_{+}}\dot{p}_{+}+\overline{T}_{+}\dot{S}_{+}
 +{q_{\mr{e}}}\dot{Z}_{+}\right)\Big|_{(\dot{\xi},y_2)},\\
 & \dot{S}_{+}(y_1,y_2)=\dot{S}_{+}(\dot{\xi},y_2)
 +\kappa\frac{1}{\overline{q}_+}\cdot\frac{\phi(\overline{T}_+)}{\overline{T}_+}{q_{\mr{e}}}\bar{Z}y_1,\\
 & \dot{Z}_{+}(y_1,y_2)=\dot{Z}_{+}(\dot{\xi},y_2)
 -\kappa\frac{\phi(\overline{T}_+)}{\overline{q}_+}\bar{Z}y_1.
\end{align*}
Then we obtain the following lemma:
\begin{lem}\label{lem: 3.4}
Let $\alpha \in (0,1)$ and $\beta > 2$.
If there exists $\dot{\xi} \in (0,L)$ such that \eqref{eq:2.43} holds, then there exists a solution $\left(\dot{U}_{-},\dot{U}_{+};\dot{\psi}_s', \dot{\xi} \right)$ to Problem \ref{problem: linearized}, which satisfies the estimate
\begin{equation}
\norm{\dot{U}_{-}}_{\mcc^{2,\alpha}(\dot{\Omega}_{-})}
+\norm{\dot{U}_{+}}_{(\dot{\Omega}_{+};\dot{\Gamma}_{s})}
+\norm{\dot{\psi}_s'}_{W_{\beta}^{1-1/\beta}(\dot{\Gamma}_{s})}
 \leq \dot{C}(\sigma+\kappa),
\end{equation}
where the constant $\dot{C}$ depending on $\overline{U}_{\pm}, L, \dot{\xi}$, $\alpha$ and $\beta$.
\end{lem}

\subsection{The Approximating Location of The Shock Front}
To estimate the approximating location of the shock front as $\sigma$ and $\kappa$ tend to zero, we consider four cases as Theorem \ref{thm: apprximating location of shock-front}. Here we give its proof:

\begin{proof} [Proof of Theorem \ref{thm: apprximating location of shock-front}]
	For statement (i), if $\kappa = A_1 \sigma^s$, denote $\bar{\xi}_*$ is the solution of $R_\sigma(\xi)= P_{\sigma *}$ since $P_{\sigma *} \in (\underline{R_\sigma}, \overline{R_\sigma})$. Denote the functional
	\begin{align*}
		I(\xi; \sigma, \kappa)= \sigma R_\sigma(\xi) + \kappa R_\kappa(\xi)
		-\sigma P_{\sigma *} - \kappa P_{\kappa *} .
	\end{align*}
	Thus clearly, $I(\bar{\xi}_*; \sigma, 0)=0.$ Taking derivative of $\xi$,
	\begin{align*}
		\frac{\partial I}{\partial \xi}(\bar{\xi}_*; \sigma ,\kappa)
		=\sigma R'_\sigma(\bar{\xi}_*) + \kappa R'_\kappa (\bar{\xi}_*)
		=-\sigma \dot{K}_1 \Theta(\bar{\xi}_*) + \kappa \dot{K}_2.
	\end{align*}
	Since $\Theta(\bar{\xi}_*) \neq 0$, $\displaystyle \frac{\partial I_\kappa}{\partial \xi}(\bar{\xi}_*; \sigma , \kappa) \neq 0$ if $\sigma_1^{s-1} < \displaystyle \frac{\dot{K}_1 |\Theta(\bar{\xi}_*)|}{A_1 \dot{K}_2}$.
	Then the implicit function theorem implies the existence of the solution $\dot{\xi}$ to the equation \eqref{R=P}.
	Moreover, from the expansion
	\begin{align*}
		0=I(\dot{\xi}; \sigma, \kappa)
		&= \kappa (R_\kappa(\bar{\xi}_*)- P_{\kappa *}) - \sigma \dot{K}_1 \Theta(\bar{\xi}_*) (\dot{\xi}-\bar{\xi}_*)
		+ \kappa \dot{K}_2  (\dot{\xi}-\bar{\xi}_*) + O(1) \sigma (\dot{\xi}-\bar{\xi}_*)^2,
	\end{align*}
	we can get the estimate
	\begin{align*}
		|\dot{\xi}-\bar{\xi}_*|	\leq \frac{\kappa |R_\kappa(\bar{\xi}_*)-P_{\kappa *}|}{ \sigma \dot{K}_1 |\Theta(\bar{\xi}_*)| - \kappa \dot{K}_2} \leq C_1 \sigma^{s-1},
	\end{align*}
	which complete the proof of (i).
	
	Similarly, if $\sigma=A_2 \kappa^s$, denote $\bar{\xi}_*$ is the solution of $R_\kappa(\xi)= P_{\kappa *}$ since $P_{\kappa *} \in (\underline{R_\kappa}, \overline{R_\kappa})$ and
	denote the functional
	\begin{align*}
		I(\xi; \sigma, \kappa)= \sigma R_\sigma(\xi) + \kappa R_\kappa(\xi)
		-\sigma P_{\sigma *} - \kappa P_{\kappa *} .
	\end{align*}
	Then $I (\bar{\xi}_*; 0, \kappa)=0$ and 
	$\displaystyle \frac{\partial I}{\partial \xi}(\bar{\xi}_*; \sigma ,\kappa) \neq 0$ if $\kappa_1^{s-1} < \displaystyle \frac{\dot{K}_2}{A_2 \dot{K}_1 \Theta(\bar{\xi}_*)}$.
	Using the same method in the proof of (i), we also have the existence of $\dot{\xi}$ and the estimate
	\begin{align*}
		|\dot{\xi}-\bar{\xi}_*|	\leq \frac{\sigma |R_\sigma(\bar{\xi}_*)- P_{\sigma *}|}{- \sigma \dot{K}_1 \Theta(\bar{\xi}_*) + \kappa \dot{K}_2} \leq C_2 \kappa^{s-1},
	\end{align*}
	for $0<\sigma < \sigma_1$. 
	Therefore the statement (ii) holds.
	
	If $\sigma = A \kappa $, we have
	\begin{align*}
		R(\xi; \sigma, \kappa)= \sigma R_{\sigma}(\xi) + \kappa R_{\kappa}(\xi)= \kappa (A R_\sigma + R_\kappa) = \kappa R_A.
	\end{align*}
	Then the solvability condition \eqref{R=P} becomes $R_A(\xi) = P_{A*}$ since $P_*= \kappa P_{A*}.$
	By the continuity of $R_A$, if $P_A$ satisfies $P_{A*} \in (\underline{R_A}, \overline{R_A})$,
	there exists $\bar{\xi}_* \in (0,L)$ such that $R_A(\bar{\xi}_*) = P_{A*}$. It implies (iii).
	
	For statement (iv), by the definition of $\mathcal{C}$, there exists $\bar{\xi}_* \in (0,L)$ such that 
	\begin{align} \label{R=P v2}
		R_{\sigma}(\bar{\xi}_*)= P_{\sigma*}, \quad
		R_{\kappa}(\bar{\xi}_*)= P_{\kappa*}.
	\end{align}
	It implies that \eqref{R=P} holds for $\dot{\xi}(\sigma, \kappa) \equiv \bar{\xi}_*$.
	
\end{proof}

Combining Theorem \ref{thm: apprximating location of shock-front} and Lemma \ref{lem: 3.4}, it is obvious that Theorem \ref{thm: linearized solution} holds.

\begin{figure}[htbp]
	\centering
	\includegraphics[height=7.0cm, width=10cm]{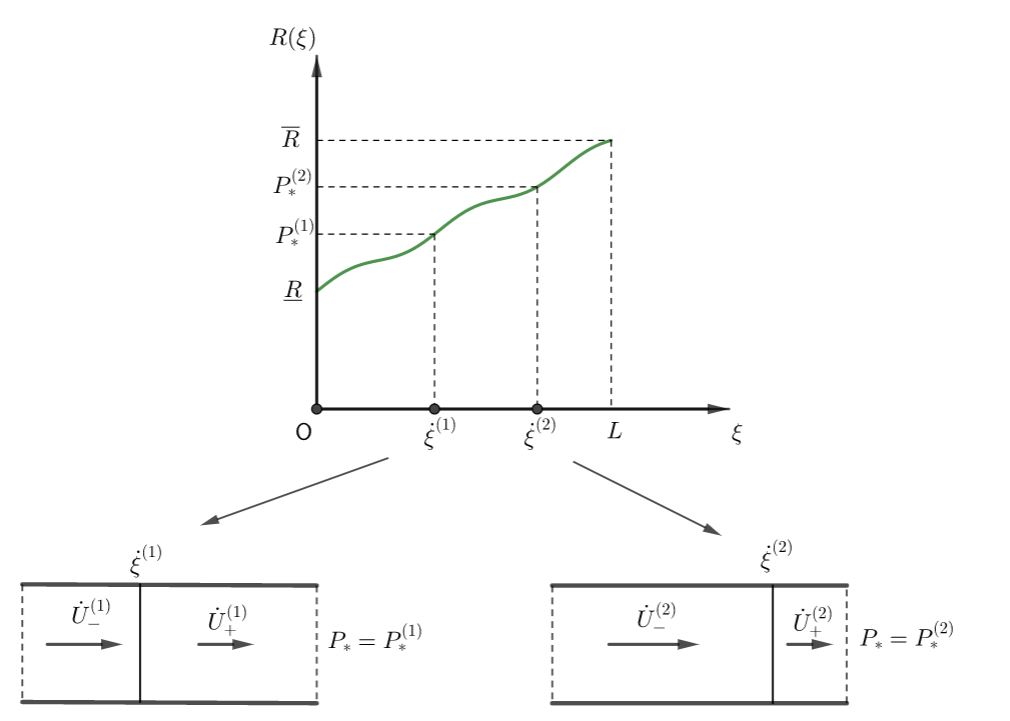}%
	\caption{the location of shock front in case (ii) of Theorem \ref{thm: apprximating location of shock-front}}
	\label{fig:approximating location}
\end{figure}

\begin{rem}
	For general cases, given an admissible pressure at the exit, there may exist multiple solutions for $\dot{\xi}$ satisfying the solvability condition \eqref{condition: admissible pressure}. However, in case (ii) of Theorem \ref{thm: apprximating location of shock-front}, $R(\xi)$ is a strictly increasing function of $\xi$ as $\sigma$ and $\kappa$ are sufficiently small. Therefore the approximating location of the shock front $\dot{\xi}$ can be determined uniquely for any admissible $P_\mr{e}$ satisfying $P_* \in (\underline R, \overline R)$ and then we can get a unique solution $\left(\dot{U}_{-},\dot{U}_{+};\dot{\psi}_s', \dot{\xi} \right)$ to Problem \ref{problem: linearized}.
	Moreover, if we change the pressure at the exit from $P^{(1)}$ to $P^{(2)}$ such that the corresponding quantity $P_*^{(1)}$ increases to $P_*^{(2)} < \bar{R}$, then the approximating location of the shock front gets closer to the exit, i.e. $\dot{\xi}^{(1)} < \dot{\xi}^{(2)} < L$ (See Figure \ref{fig:approximating location}).
	In case (i), $R(\xi)$ is a strictly monotone function if the nozzle is expanding or contracting and $\sigma$ and $\kappa$ are sufficiently small. Therefore we can get the similar results with the same argument in case (ii).
\end{rem}

\section{The Nonlinear Iteration}

In this section, we will construct the iteration mapping and use the background solution and the initial approximation location of the shock front as the initial approximation.

\subsection{The upstream flow}
Since the Euler system in the supersonic region is hyperbolic, according to the hyperbolic theory, we can get a supersonic flow in the whole nozzle without a shock as $\sigma$ and $\kappa$ are sufficiently small.

\begin{thm}
	Suppose \eqref{perturb4} holds. Then there exist constants $\delta_L$ and $C_L$ depending on $\bar{U}_-$ and $L$, such that for any $0 \leq \sigma + \kappa \leq \delta_L$, the problem
	\begin{align}
	\label{super1}&\text{The reacting Euler system $\eqref{LEuler1}-\eqref{LEuler5}$},& &\text{in $\Omega$}, \\
	\label{super2}&U_-(0,y_2)=\bar{U}_-, & &\text{on $\Gamma_1$}, \\
	\label{super3}&\theta_-(y_1,0)=0, & &\text{on $\Gamma_2$}, \\
	\label{super4}&\theta_-(y_1,1)=\sigma \Theta(y_1),& &\text{on $\Gamma_4$},
	\end{align}
	has a unique solution $U_- \in C^{2,\alpha}(\bar{\Omega})$ such that
	\begin{equation}\label{superes1}
	\|U_- - \bar{U}_- \|_{C^{2,\alpha}(\bar{\Omega})} \leq C_L (\sigma + \kappa).
	\end{equation}
	Moreover, let $\delta U:=U_- - \bar{U}_-$. Then there exists a constant $\tilde{C}_L$ depending on $\bar{U}_-$ and $L$ such that
	\begin{equation}\label{superes2}
	\|\delta U_- - \dot{U}_-\|_{C^{1,\alpha}(\bar{\Omega})} \leq \tilde{C}_L (\sigma+\kappa)^2.
	\end{equation}
\end{thm}

\begin{proof}
	The existence of the unique solution $U_- \in C^{2,\alpha}(\bar{\Omega})$ can be obtained by applying the theory of hyperbolic system in \cite{LiYu1985}. To estimate $\|\delta U_- - \dot{U}_-\|_{C^{1,\alpha}(\bar{\Omega})}$, rewrite Euler system \eqref{LEuler1}-\eqref{LEuler5} as
	\begin{equation}
	A(U)\partial_{y_1} U + B \partial_{y_2} U = C(U),
	\end{equation}
	where
	\begin{equation*}
	A(U)=\left(
	\begin{matrix}
	-\displaystyle \frac{\cos \theta}{\rho q}\frac{1-M^2}{\rho q^2} & -\displaystyle \frac{\sin \theta}{\rho q} & 0 & 0 & 0 \\
	-\displaystyle \frac{\sin \theta}{\rho q} & q \cos \theta & 0 & 0 & 0 \\
	1 & 0 & \rho q & 0 & 0 \\
	0 & 0 & 0 & 1 & 0 \\
	0 & 0 & 0 & 0 & 1
	\end{matrix}
	\right)
	,
	\end{equation*}
	\begin{equation*}
	B=
	\left(
	\begin{matrix}
	0 & 1 & 0 & 0 & 0 \\
	1 & 0 & 0 & 0 & 0 \\
	0 & 0 & 0 & 0 & 0 \\
	0 & 0 & 0 & 0 & 0 \\
	0 & 0 & 0 & 0 & 0
	\end{matrix}
	\right)
	, \quad
	C=
	\left(
	\begin{matrix}
	\kappa f_1(U) \\ 0 \\ 0 \\ \kappa f_4(U) \\ -\kappa f_5(U)
	\end{matrix}
	\right).
	\end{equation*}
	
	Since $U_-$ satisfies \eqref{super1}-\eqref{super4} and $\dot{U}_-$ satisfies
	\begin{equation*}
	A(\bar{U}_-)\partial_{y_1} U + B \partial_{y_2} U = C(\bar{U}_-),
	\end{equation*}
	with the same boundary conditions. Hence $\delta U_- - \dot{U}_-$ satisfies the following problem:
	\begin{align*}
	&A(\bar{U}_-)\partial_{y_1} (\delta U_- - \dot{U}_-) + B \partial_{y_2} (\delta U_- - \dot{U}_-) = \tilde{C}(\delta U_-),& &\text{in $\Omega$}, \\
	&\delta U_- - \dot{U}_-=0,& &\text{on $\Gamma_1$}, \\
	&\theta_- - \dot{\theta}_-=0,& &\text{on $\Gamma_2$}, \\
	&\theta_- - \dot{\theta}_-=0,& &\text{on $\Gamma_4$},
	\end{align*}
	where
	\begin{equation*}
	\tilde{C}(\delta U_-) = (A(\bar{U}_-)-A(U_-))\partial_{y_1} \delta U_- + (C(U_-)-C(\bar{U}_-)).
	\end{equation*}
	Note that $\|C(U_-)-C(\bar{U}_-)\|_{C^{1,\alpha}(\bar{\Omega})} \leq
	C \kappa \| \delta U_- \|_{C^{1,\alpha}(\bar{\Omega})}$,
	we have
	\begin{align*}
	\|\delta U_- - \dot{U}_-\|_{C^{1,\alpha}(\bar{\Omega})} & \leq
	C \|\tilde{C}(\delta U_-)\|_{C^{1,\alpha}(\bar{\Omega})} \\
	& \leq C \| \partial_{y_1} \delta U_- \| _{C^{1,\alpha}(\bar{\Omega})}
	\| \delta U_- \|_{C^{1,\alpha}(\bar{\Omega})}
	+ C \kappa \|\delta U_-\|_{C^{1,\alpha}(\bar{\Omega})} \\
	& \leq C (\sigma + \kappa)^2,
	\end{align*}
	which proves \eqref{superes2}.
\end{proof}

\subsection{The shock front and the downstream flow}
We describe a shock-front $\Gamma_s$ whose location is close to $\dot{\Gamma}_s$ by
\begin{equation*}
\Gamma_s := \{ (y_1,y_2): y_1=\psi_s(y_2) := \dot{\xi} + \delta \psi_s(y_2), 0<y_2<1 \},
\end{equation*}
while the downstream region is
\begin{equation*}
\Omega_+ := \{ (y_1,y_2): \psi_s(y_2) < y_1 < L , 0<y_2<1 \}.
\end{equation*}
Our goal is to find a solution $(U_+; \psi)$ which is close to $(\bar{U}, \bar{\psi})$ and solves the free boundary problem
\begin{align*}
&\text{The reacting Euler system $\eqref{LEuler1}-\eqref{LEuler5}$},& & \qquad \text{in $\Omega$}, \\
&\text{Rankine-Hugoniot conditions $\eqref{LRH1}-\eqref{LRH5}$},& & \qquad \text{on $\Gamma_S$}, \\
&\theta_+(y_1,0)=0,& & \qquad \text{on $\Gamma_2^+$}, \\
&p_+(L,y_2)= \bar{p}_+ + P_\mr{e} (Y(L,y_2);\sigma,\kappa),& & \qquad \text{on $\Gamma_3$}, \\
&\theta_+(y_1,1)=\sigma \Theta(y_1),& & \qquad \text{on $\Gamma_4^+$},
\end{align*}
where $U_-$ in R-H conditions is given by the supersonic flow solved in the previous section.

Since the subsonic region is undetermined because of the free boundary $\Gamma_S$, to fix the region we need to introduce a transformation first which is defined as $(z_1,z_2)=\Phi(y_1,y_2,\psi)$:
\begin{align*}
&z_1=L+\frac{L-\dot{\xi}}{L-\psi(y_2)}(y_1-L), \\
&z_2=y_2.
\end{align*}
The transformation fix $\Gamma_S$ to the initial approximation location of the shock front.
Hence under the transformation, $\Omega_+, \Gamma_s, \Gamma_2^+, \Gamma_3 , \Gamma_4^+$ become, respectively,
\begin{align*}
& \tilde{\Omega} = \{(z_1,z_2): \dot{\xi} < z_1 < L , 0<z_2<1 \}, \\
& \tilde{\Gamma}_s = \{(z_1,z_2): z_1 = \dot{\xi} , 0<z_2<1 \}, \\
& \tilde{\Gamma}_2 = \{(z_1,z_2): \dot{\xi} < z_1 < L , z_2=0 \}, \\
& \tilde{\Gamma}_3 = \{(z_1,z_2): z_1 = L , 0<z_2<1\}, \\
& \tilde{\Gamma}_4 = \{(z_1,z_2): \dot{\xi} < z_1 < L , z_2 = 1 \}.
\end{align*}
Denote $\tilde{U}(z_1,z_2) = U_+(\Phi^{-1}(z_1,z_2; \psi))$. Then the free boundary value problem becomes a fixed boundary value problem:
\begin{align}
&A(\tilde{U}) \partial_{z_1} \tilde{U} + B \partial_{z_2} \tilde{U}= H(\tilde{U};\psi, \dot{\xi}),
& &\text{in $\tilde{\Omega}$}, \\
& \tilde{\theta} = 0, \quad& & \text{on $\tilde{\Gamma}_2$}, \\
& \tilde{p} = \bar{p}_+ + P_\mr{e} (Y(L,z_2);\sigma,\kappa),& & \text{on $\tilde{\Gamma}_3$}, \\
& \tilde{\theta} = \sigma \Theta (\Phi^{-1}(z_1,z_2;\psi)),& & \text{on $\tilde{\Gamma}_4$}, \\
& G_j(\tilde{U}, U_-^{(\psi',\dot{\xi})})=0, \quad j=1,2,3,4,& & \text{on $\tilde{\Gamma}_s$}, \\
& G_5(\tilde{U}, U_-^{(\psi',\dot{\xi})};\psi')=0,& &\text{on $\tilde{\Gamma}_s$},
\end{align}
where
\begin{align*}
	H(\tilde{U};\psi, \dot{\xi}):= C(\tilde{U})
	+ \frac{\dot{\xi} - \psi (z_2)}{L- \psi(z_2)} A(\tilde{U}) \partial_{z_1}\tilde{U}
	+\frac{L-z_1}{L-\psi(z_2)} B \psi ' (z_2) \partial_{z_1} \tilde{U}.
\end{align*}

Next we shall construct a nonlinear iteration scheme to prove the existence of the solution of this problem.
In the remaining part, we will drop " $\tilde{}$ " in the following argument if there is no confusion.

\subsection{The iteration scheme in the subsonic region}
For a given approximating subsonic state $U=\bar{U}_+ + \delta U$ and a given approximating location of the shock front $\psi(z_2)=\dot{\xi} + \delta \xi - \int_{z_2}^1 \delta \psi'(s)ds$,
where $\delta U$ and $\delta \psi'$ are close to $\dot{U}_+$ and $\dot{\psi}'$, respectively, we construct an iteration mapping
\begin{equation*}
	J: (\delta U;\delta \xi , \delta \psi') \mapsto (\delta U ^*;\delta \xi ^* , (\delta \psi^*)'),
\end{equation*}
according to the linearized problem:

In the subsonic region $\tilde{\Omega}$,
$\delta U^* = (\delta p^*, \delta \theta^*, \delta q^*, \delta S^*, \delta Z^*)$ and
$\delta \xi^*$ satisfy the linearized Euler equations:

\begin{align}
\label{it1} -\frac{1}{\bar{\rho}_+ \bar{q}_+} \frac{1-\bar{M}_+^2}{\bar{\rho}_+\bar{q}_+^2} \partial_{z_1}(\delta p^*)
+\partial_{z_2}(\delta \theta^*) &= \tilde{f}_1(\delta U; \delta \psi', \delta \xi^*), \\
\label{it2} \bar{q}_+ \partial_{z_1}(\delta \theta^*) + \partial_{z_2}(\delta p^*)
&= \tilde{f}_2 (\delta U; \delta \psi', \delta \xi^*),\\
\label{it3} \partial_{z_1} (\bar{q}_+ \delta q^* + \frac{1}{\bar{\rho}_+} \delta p^* + \bar{T}_+ \delta S^* +q_{\mr{e}}\delta Z^*)
&= \partial_{z_1} \tilde{f}_3 (\delta U; \delta \psi', \delta \xi^*),\\
\label{it4} \partial_{z_1} (\delta S^*) &= \tilde{f}_4 (\delta U; \delta \psi', \delta \xi^*),\\
\label{it5} \partial_{z_1} (\delta Z^*) &= \tilde{f}_5 (\delta U; \delta \psi', \delta \xi^*),
\end{align}
where
\begin{align*}
\begin{split}
\tilde{f}_1(\delta U; \delta \psi', \delta \xi^*) &= \kappa \bar{f}_1^+
+ \frac{\sin \theta}{\rho q}\partial_{z_1} \theta +
\left( \frac{\cos \theta}{\rho q}\frac{1-M_+^2}{\rho q^2}
--\frac{1}{\bar{\rho}_+ \bar{q}_+} \frac{1-\bar{M}_+^2}{\bar{\rho}_+\bar{q}_+^2} \right)
\partial_{z_1} p \\
&\quad + \kappa (f_1^+(\delta U, \delta \psi') - \bar{f}_1^+) + \frac{L-z_1}{L-\psi(z_2)} \psi'(z_2) \partial_{z_1} \theta \\
&\quad -\frac{\dot{\xi} - \psi(z_2)}{L-\psi(z_2)}
\left( \frac{\cos \theta}{\rho q}\frac{1-M_+^2}{\rho q^2} \partial_{z_1} p
+ \frac{\sin \theta}{\rho q} \partial_{z_1} \theta \right),
\end{split} \\
\begin{split}
\tilde{f}_2(\delta U; \delta \psi', \delta \xi^*) &=
-\frac{\sin \theta}{\rho q} \partial_{z_1}p + (q\cos\theta-\bar{q}_+) \partial_{z_1} \theta \\
&\quad +\frac{\dot{\xi} - \psi(z_2)}{L-\psi(z_2)}
\left ( q \cos \theta \partial_{z_1} \theta - \frac{\sin \theta}{\rho q} \partial_{z_1} p \right) \\
&\quad +  \frac{L-z_1}{L-\psi(z_2)} \psi'(z_2) \partial_{z_1}p,
\end{split} \\
\tilde{f}_3(\delta U; \delta \psi', \delta \xi^*) &=
(\bar{q}_+ \delta q + \frac{1}{\bar{\rho}_+} \delta p + \bar{T}_+ \delta S +q_{\mr{e}}\delta Z)\
-(\frac{1}{2}q^2 +i+q_{\mr{e}} Z), \\
\tilde{f}_4(\delta U; \delta \psi', \delta \xi^*) &= \kappa \bar{f}_4^+
+ \kappa (f_4^+(\delta U, \delta \psi') - \bar{f}_4^+)
+ \frac{\dot{\xi} - \psi(z_2)}{L-\psi(z_2)} \partial_{z_1} (\delta S), \\
\tilde{f}_5(\delta U; \delta \psi', \delta \xi^*) &= -\kappa \bar{f}_5^+
- \kappa (f_5^+(\delta U, \delta \psi') - \bar{f}_5^+)
+ \frac{\dot{\xi} - \psi(z_2)}{L-\psi(z_2)} \partial_{z_1} (\delta Z).
\end{align*}

On the boundaries, $\delta U^*$ and $\delta \xi^*$ satisfy
\begin{align}
\label{itb1} &\delta \theta^* (z_1,0) = 0,  & \text{on $\tilde{\Gamma}_2$}, \\
\label{itb2} &\delta p^*(L,z_2) = \delta P(z_2; \delta U),  & \text{on $\tilde{\Gamma}_3$}, \\
\label{itb3} &\delta \theta^* (z_1,1) = \delta \Theta_4(z_1; \delta \xi^*), & \text{on $\tilde{\Gamma}_4$},
\end{align}
where
\begin{align*}
&\delta P(z_2; \delta U):= P_\mr{e}(Y(L, z_2; \delta U);\sigma,\kappa),\\
&Y(L, z_2; \delta U)=\int_0^{z_2} \frac{1}{(\rho q \cos \theta)(L,s)}ds, \\
&\delta \Theta_4(z_1; \delta \xi^*) = \sigma \Theta\left(z_1+ \frac{\delta \xi^*}{L-\dot{\xi}}(L-z_1)\right).
\end{align*}

On the shock-front $\tilde{\Gamma}_s$, $\delta U^*$, $\delta \xi^*$ and $(\delta \psi^*)'$ satisfy the linearized Rankine-Hugoniot conditions:
\begin{align}
\label{its1} &\beta_j^+ \cdot \delta U^* = g_j(\delta U, \delta U_-; \delta \psi', \delta \xi^*), \quad j=1,2,3,4, \\
\label{its2} &\beta_5^+ \cdot \delta U^* - [\bar{p}](\delta \psi^*)' = g_5 (\delta U, \delta U_-; \delta \psi', \delta \xi^*),
\end{align}
where
\begin{align*}
&g_j(\delta U, \delta U_-; \delta \psi', \delta \xi^*) =
\beta_j^+ \cdot \delta U - G_j(U, U_-^{(\psi',\xi_*)}), \quad j=1,2,3,4, \\
& g_5 (\delta U, \delta U_-; \delta \psi', \delta \xi^*) =
\beta_5^+ \cdot \delta U - [\bar{p}]\delta \psi' - G_5(U, U_-^{(\psi',\xi_*)};\psi').
\end{align*}
Similar as the previous section, we can rewrite the R-H conditions into
\begin{equation}
\left(
\begin{matrix}
\delta p^* \\ \delta q^* \\ \delta S^* \\ \delta Z^*
\end{matrix}
\right)
=
\left(
\begin{matrix}
g_1^\sharp \\ g_2^\sharp \\ g_3^\sharp \\ g_4^\sharp
\end{matrix}
\right)
:=
B_s^{-1}
\left(
\begin{matrix}
g_1 \\ g_2 \\ g_3 \\ g_4
\end{matrix}
\right),
\quad \text{on $\Gamma_s$},
\end{equation}
where
$g_j^\sharp = g_j^\sharp (\delta U, \delta U_-; \delta \psi', \delta \xi^*), j=1,2,3,4$.

Define
\begin{equation*}
\Sigma(\varepsilon):=
\{ (\delta U, \delta \psi'): \| \delta U \|_{(\Omega, \Gamma_s)}
+ \| \delta \psi' \|_{W_{\beta}^{1-\frac{1}{\beta}}(\Gamma_s)} \leq \varepsilon \},
\end{equation*}
and
\begin{equation*}
\mathcal{S}(\varepsilon; \dot{U}_+, \dot{\psi}'):=
\{ (\delta U, \delta \psi'):
(\delta U - \dot{U}_+ , \delta \psi'-\dot{\psi}') \in \Sigma({\varepsilon})
\}.
\end{equation*}

Next we will prove the iteration mapping
\begin{align*}
	J' : \mathcal{S}(\displaystyle \frac{1}{2} (\sigma + \kappa)^{3/2}; \dot{U}_+, \dot{\psi}')
	&\to \mathcal{S}(\displaystyle \frac{1}{2} (\sigma + \kappa)^{3/2}; \dot{U}_+, \dot{\psi}'), \\
	(\delta U, \delta \psi') &\mapsto (\delta U ^*, (\delta \psi^*)'),
\end{align*}
is well-defined for sufficiently small $\sigma$ and $\kappa$.
The proof is divided into several steps:
\begin{enumerate}
	\item Determine $\delta \xi^*$ and solve for $(\delta p^*, \delta \theta^*)$ by \eqref{it1},\eqref{it2},\eqref{itb1}-\eqref{itb3} and \eqref{its1};
	\item Solve for $\delta S^*$ and $\delta Z^*$ by \eqref{it4}, \eqref{it5} and \eqref{its1};
	\item Solve for $\delta q^*$ by \eqref{it3} and \eqref{its1};
	\item Update $(\delta \psi)'$ by \eqref{its2}.
\end{enumerate}
And then we will prove that the iteration mapping $J'$ is contractive so that we can get a fixed point via the contraction mapping principle. The fixed point is the solution to Problem \ref{problem: Lagrange}.

\subsubsection{Determine $\delta \xi^*$ and solve for $(\delta p^*, \delta \theta^*)$}
Consider the equations
\begin{align}
\label{pt1} &-\frac{1}{\bar{\rho}_+ \bar{q}_+} \frac{1-\bar{M}_+^2}{\bar{\rho}_+\bar{q}_+^2} \partial_{z_1}(\delta p^*)
+\partial_{z_2}(\delta \theta^*) = \tilde{f}_1(\delta U; \delta \psi', \delta \xi^*),
& &\text{in $\tilde{\Omega}$},\\
\label{pt2}&\bar{q}_+ \partial_{z_1}(\delta \theta^*) + \partial_{z_2}(\delta p^*)
= \tilde{f}_2 (\delta U; \delta \psi', \delta \xi^*),
& & \text{in $\tilde{\Omega}$},\\
\label{pt3}&\delta \theta^* (z_1,0) = 0, & & \text{on $\tilde{\Gamma}_2$}, \\
\label{pt4}&\delta p^*(L,z_2) = \delta P(z_2; \delta U), & & \text{on $\tilde{\Gamma}_3$}, \\
\label{pt5}&\delta \theta^* (z_1,1) = \delta \Theta_4(z_1; \delta \xi^*),& & \text{on $\tilde{\Gamma}_4$},\\
\label{pt6}&\delta p^*(\bar{\xi_*},z_2) = g_1^{\sharp} (\delta U, \delta U_-; \delta \psi', \delta \xi^*),
& &\text{on $\tilde{\Gamma}_s$}.
\end{align}
Define the functional:
\begin{equation}
\begin{split}
I(\delta \xi^*, \delta U, \delta \psi', \delta U_-, c)
= &-\int_{\tilde{\Omega}} \tilde{f}_1(\delta U; \delta \psi',\delta \xi^*) dz_1 dz_2
+\int_{\dot{\xi}}^L \delta \theta_4 (z_1 ; \delta \xi^*) dz_1\\
&-\frac{1}{\bar{\rho}_+ \bar{q}_+} \frac{1-\bar{M}_+^2}{\bar{\rho}_+\bar{q}_+^2}
\int_0^1 \delta P(z_2;\delta U) dz_2 \\
& +\frac{1}{\bar{\rho}_+ \bar{q}_+} \frac{1-\bar{M}_+^2}{\bar{\rho}_+\bar{q}_+^2}
\int_0^1 g_1^\sharp(\delta U, \delta U_- ; \delta \psi', \delta \xi^*) dz_2
+c.
\end{split}
\end{equation}
According to the lemma of solvability, the equations has a solution if and only if
\begin{equation}\label{sol}
I(\delta \xi^*, \delta U, \delta \psi', \delta U_-, 0)=0.
\end{equation}
By the calculation of the initial approximating location of the shock front, we already have
\begin{equation}
I(0;0,0;\dot{U}_- , \dot{c})=0,
\end{equation}
where
\begin{equation*}
\dot{c}=\frac{1}{\bar{\rho}_+ \bar{q}_+} \frac{1-\bar{M}_+^2}{\bar{\rho}_+\bar{q}_+^2}
\int_{\tilde{\Gamma}_s} (-g_1^{\sharp} (0,\dot{U}_1;0,0) + \dot{g}_1^{\sharp} (0,\dot{U}_1;0,0)) dz_2.
\end{equation*}

Next, we will expand $I$ near $(0;0,0;\dot{U}_-, \dot{c})$ and solve $\delta \xi^*$ first by the implicit theorem.

\begin{thm} \label{sol thm}
	There exists small constants $\sigma_1$ and $\kappa_1$ depending on $\dot{\xi}$ such that
	for any $0<\sigma \leq \sigma_1$ and $0< \kappa \leq \kappa_1$ satisfy \eqref{admissble condition},
	if $(\delta U, \delta \psi') \in \mathcal{S}(\displaystyle \frac{1}{2} (\sigma + \kappa)^{3/2}; \dot{U}_+, \dot{\psi}')$,
	there exists a unique solution $\delta \xi^*$ to the equation \eqref{sol} with the estimate
	$|\delta \xi^*| \leq C_* (\sigma+\kappa)$, where the constant $C_*$ depends on $\dot{\xi}$
	and $\beta_0$.
\end{thm}

\begin{proof}
	To expand $I$, we shall estimate each term first. Firstly, we separate
	$\int_{\tilde{\Omega}} \tilde{f}_1 dz_1 dz_2$ into six parts:
	\begin{equation*}
	\int_{\tilde{\Omega}} \tilde{f}_1 (\delta U;\delta \psi', \delta \xi^*) dz_1 dz_2
	=\int_{\tilde{\Omega}} \kappa \bar{f}_1 dz_1 dz_2
	+\sum_{i=1}^5 I_i,
	\end{equation*}
	where
	\begin{align*}
	&I_1 = \int_{\Omega} \kappa (f_1^+ - \bar{f}_1^+) dz_1 dz_2, \\
	&I_2 = \int_{\Omega} \frac{\sin \theta}{\rho q} \partial_{z_1} \theta dz_1 dz_2, \\
	&I_3 = \int_{\Omega} \left( \frac{\cos \theta}{\rho q} \frac{1-M_+^2}{\rho q^2}
	- \frac{1}{\bar{\rho}_+ \bar{q}_+} \frac{1-\bar{M}_+^2}{\bar{\rho}_+  \bar{q}_+^2} \right)
	\partial_{z_1}p dz_1 dz_2, \\
	&I_4 = \int_{\Omega} \frac{\delta \xi^* - \int_{z_2}^1 \delta \psi'(s) ds}{L-\psi(z_2)}
	\left( \frac{\cos \theta}{\rho q} \frac{1-M_+^2}{\rho q^2} \partial_{z_1} p
	+\frac{\sin \theta}{\rho q} \partial_{z_1}\theta \right) dz_1 dz_2, \\
	&I_5 = \int_{\Omega} \frac{L-z_1}{L-\psi(z_2)} \psi'(z_2) \partial_{z_1} \theta dz_1 dz_2 .
	\end{align*}
	Since $(\delta U, \delta \psi') \in \Sigma(\kappa + \sigma)$, it is easy to estimate that
	\begin{align*}
	& I_1 =  O(1) \kappa (\kappa + \sigma ), \\
	& I_i =  O(1) (\kappa+\sigma)^2, \quad i=2,3,5. \\
	\end{align*}
	For $I_4$, we have
	\begin{align*}
	I_4 &= \delta \xi^* \int_{\Omega} \frac{1}{L-\psi(z_2)} \frac{\cos \theta}{\rho q} \frac{1-M_+^2}{\rho q^2} \partial_{z_1} p dz_1 dz_2 \\
	& \quad +\delta \xi^* \int_{\Omega} \frac{1}{L-\psi(z_2)} \frac{\sin \theta}{\rho q} \partial_{z_1}\theta dz_1 dz_2 \\
	& \quad -\int_\Omega  \frac{\int_{z_2}^1 \delta \psi'(s) ds}{L-\psi(z_2)}
	\left( \frac{\cos \theta}{\rho q} \frac{1-M_+^2}{\rho q^2} \partial_{z_1} p
	+\frac{\sin \theta}{\rho q} \partial_{z_1}\theta \right) dz_1 dz_2 \\
	&=\delta \xi^* \int_{\Omega} \frac{1}{L-\dot{\xi}} \frac{1}{\bar{\rho}_+ \bar{q}_+}
	\frac{1-\bar{M}_+^2}{\bar{\rho}_+ \bar{q}_+^2} \partial_{z_1} \dot{p} dz_1 dz_2
	+\delta \xi^* O(1) (\sigma+\kappa)^{\frac{3}{2}} \\
	& \quad + \delta \xi^* O(1) (\sigma+\kappa)^2 +(\delta \xi^*)^2 O(1) (\sigma +\kappa)
	+O(1) (\sigma + \kappa)^2.
	\end{align*}
	From the equation
	\begin{equation*}
	-\frac{1}{\bar{\rho}_+ \bar{q}_+}
	\frac{1-\bar{M}_+^2}{\bar{\rho}_+ \bar{q}_+^2}
	\partial_{y_1} \dot{p}_+ + \partial_{y_2} \dot{\theta}_+ = \kappa \bar{f}_1^+,
	\end{equation*}
	we know
	\begin{align*}
	\frac{1}{\bar{\rho}_+ \bar{q}_+}
	\frac{1-\bar{M}_+^2}{\bar{\rho}_+ \bar{q}_+^2}
	\int_{\Omega} \partial_{z_1} \dot{p}_+ dz_1 dz_2
	&=\int_{\Omega} (-\kappa \bar{f}_1^+ + \partial_{z_2} \dot{\theta}_+) dz_1 dz_2 \\
	&= -\kappa \bar{f}_1^+ (L-\dot{\xi}) + \int_{\dot{\xi}}^L \sigma \Theta(z_1) dz_1.
	\end{align*}
	Hence
	\begin{align}
	\begin{split}
	I_4 &= \delta \xi^* \left(-\kappa \bar{f}_1^+ + \sigma \frac{1}{L-\bar{\xi_*}}
	\int_{\dot{\xi}}^L \Theta (z_1) \right)
	+\delta \xi^* O(1) (\sigma+\kappa)^{\frac{3}{2}} \\
	& \quad + \delta \xi^* O(1) (\sigma+\kappa)^2 +(\delta \xi^*)^2 O(1) (\sigma +\kappa)
	+O(1) (\sigma + \kappa)^2.
	\end{split}
	\end{align}
	
	Next, the second and third integral of $I$ can be written as
	\begin{align*}
	\int_{\dot{\xi}}^L \delta \theta_4 (z_1 ; \delta \xi^*) dz_1
	&= \sigma \int_{\dot{\xi}}^L \Theta (z_1+\frac{\delta \xi^*}{L-\dot{\xi}} (L-z_1)) dz_1\\
	&= \sigma \int_{\dot{\xi}+\delta \xi^*}^L \Theta(y_1) (1+\frac{\delta \xi^*}{L-\dot{\xi}})dy_1 \\
	&= \sigma \int_{\dot{\xi}}^L \Theta(y_1) dy_1
	- \sigma \int_{\dot{\xi}}^{\dot{\xi}+\delta \xi^*} \Theta(y_1) dy_1\\
	&\quad + \sigma \int_{\dot{\xi}+\delta \xi^*}^L \Theta(y_1) \frac{\delta \xi^*}{L-\dot{\xi}}dy_1 \\
	&= \sigma \int_{\dot{\xi}}^L \Theta(y_1) dy_1
	- \delta \xi^* \left( \sigma \Theta(\dot{\xi}) -
	\sigma \frac{1}{L-\dot{\xi}} \int_{\dot{\xi}}^L \Theta(y_1) dy_1 \right) \\
	&\quad + O(1) \sigma (\delta \xi^*)^2,
	\end{align*}
and
	\begin{align*}
	\frac{1}{\bar{\rho}_+ \bar{q}_+} \frac{1-\bar{M}_+^2}{\bar{\rho}_+\bar{q}_+^2}
	\int_0^1 \delta P(z_2;\delta U) dz_2
	&= \frac{1}{\bar{\rho}_+ \bar{q}_+} \frac{1-\bar{M}_+^2}{\bar{\rho}_+\bar{q}_+^2}
	\int_0^1 P_\mr{e}(z_2) dz_2 \\
	& \quad +\frac{1}{\bar{\rho}_+ \bar{q}_+} \frac{1-\bar{M}_+^2}{\bar{\rho}_+\bar{q}_+^2}
	\int_0^1 P_\mr{e}(Y(L,z_2;\delta U))-P_\mr{e}(z_2) dz_2 \\
	&= \frac{1}{\bar{\rho}_+ \bar{q}_+} \frac{1-\bar{M}_+^2}{\bar{\rho}_+\bar{q}_+^2}
	\int_0^1 P_\mr{e}(z_2) dz_2 + O(1) (\sigma + \kappa)^2.
	\end{align*}
	
	To estimate $g_1^\sharp$, we consider $g_j(j=1,2,3,4)$ first:
	\begin{align*}
	g_j &= \beta_j^+ \delta U - G_j (U, U_-^{(\psi', \xi_*)}) \\
	&= \left(\beta_j^+ \delta U + \beta_j^- \delta U_-^{(\psi', \xi_*)}
	- G_j (U, U_-^{(\psi', \xi_*)})\right) - \beta_j^- \delta U_-^{(\psi', \xi_*)} \\
	&= \frac{1}{2} \int_0^1 D^2 G_j ( \bar{U}_+ + s\delta U; \bar{U}_- + s \delta U_-) ds (\delta U, \delta U_-)^2
	- \beta_j^- (\delta U_-^{(\psi', \xi_*)}- \dot{U}_-(\psi(z_2),z_2)) \\
	& \quad - \beta_j^- (\dot{U}_-(\psi(z_2),z_2)-\dot{U}_- (\xi_*, z_2))
	- \beta_j^- \dot{U}_- (\dot{\xi} + \delta \xi^* , z_2) \\
	&= - \beta_j^- \dot{U}_- (\dot{\xi} + \delta \xi^* , z_2) + O(1)(\kappa + \sigma)^2.
	\end{align*}
	Hence
	\begin{align*}
	g_j^\sharp (\delta U, \delta U_- ; \delta \psi', \delta \xi^*)\mid_{(\dot{\xi},z_2)} =
	\dot{g}_j^\sharp \mid_{(\dot{\xi} + \delta \xi^* ,z_2)} + O(1) (\kappa + \sigma)^2,
	\end{align*}
	and
	\begin{align*}
	\dot{c}=\frac{1}{\bar{\rho}_+ \bar{q}_+} \frac{1-\bar{M}_+^2}{\bar{\rho}_+\bar{q}_+^2}
	\int_{\tilde{\Gamma}_s} (-g_1^{\sharp} (0,\dot{U}_1;0,0) + -\dot{g}_1^{\sharp} (0,\dot{U}_1;0,0)) dz_2
	=O(1)(\kappa + \sigma)^2.
	\end{align*}
	Integrate $g_1^\sharp$ by $z_2$ from 0 to 1 to get
	\begin{align*}
	&\quad \frac{1}{\bar{\rho}_+ \bar{q}_+} \frac{1-\bar{M}_+^2}{\bar{\rho}_+\bar{q}_+^2}
	\int_0^1 g_1^\sharp(\delta U, \delta U_- ; \delta \psi', \delta \xi^*) dz_2 \\
	&= \frac{1}{\bar{\rho}_+ \bar{q}_+} \frac{1-\bar{M}_+^2}{\bar{\rho}_+\bar{q}_+^2}
	\int_0^1 \dot{g}_1^\sharp(\dot{\xi}, z_2) dz_2 \\
	&\quad + \frac{1}{\bar{\rho}_+ \bar{q}_+} \frac{1-\bar{M}_+^2}{\bar{\rho}_+\bar{q}_+^2}
	\int_0^1 \dot{g}_1^\sharp(\dot{\xi}+\delta \xi^*,z_2) -\dot{g}_1^\sharp(\dot{\xi}, z_2) dz_2
	+O(1) (\kappa + \sigma)^2 \\
	&= \frac{1}{\bar{\rho}_+ \bar{q}_+} \frac{1-\bar{M}_+^2}{\bar{\rho}_+\bar{q}_+^2}
	\int_0^1 \dot{g}_1^\sharp(\dot{\xi}, z_2) dz_2 \\
	& \quad + (1-\dot{K}_1) \left( \int_{\dot{\xi}}^{\dot{\xi} + \delta \xi} \sigma \Theta(z_1) dz_1
	-\kappa \bar{f}_1^- \delta \xi^* \right)\\
	& \quad -\frac{1}{\gamma C_v} \frac{1}{\bar{\rho}_- \bar{q}_-} (\frac{\bar{\rho}_+}{\bar{p}_+} \frac{\bar{p}_-}{\bar{\rho}_-} -1 + \dot{K}_1) \kappa \bar{f}_4^- \delta \xi^*
	+O(1)(\kappa + \sigma)^2 \\
	&= \frac{1}{\bar{\rho}_+ \bar{q}_+} \frac{1-\bar{M}_+^2}{\bar{\rho}_+\bar{q}_+^2}
	\int_0^1 \dot{g}_1^\sharp(\dot{\xi}, z_2) dz_2 \\
	&\quad +\delta \xi^* \left( (1-\dot{K}_1) \sigma \Theta(\bar{\xi_*})
	-\kappa \frac{\bar{\rho}_+}{\bar{p}_+} \frac{\bar{p}_-}{\bar{\rho}_-} \bar{f}^-_1
	\right) \\
	&\quad +O(1)(\kappa+\sigma)(\delta \xi^*)^2 + O(1) (\kappa + \sigma)^2.
	\end{align*}
	
	Combining the above calculation of each term of $I$, we have
	\begin{align*}
	& \quad I(\delta \xi^*, \delta U, \delta \psi', \delta U_-, 0) \\
	&= \int_{\tilde{\Omega}} \kappa \bar{f}_1 dz_1 dz_2
	+ \sigma \int_{\dot{\xi}}^L \Theta(z_1) dz_1
	- \frac{1}{\bar{\rho}_+ \bar{q}_+} \frac{1-\bar{M}_+^2}{\bar{\rho}_+\bar{q}_+^2}
	\int_0^1 P_\mr{e}(z_2) dz_2\\
	&\quad+ \frac{1}{\bar{\rho}_+ \bar{q}_+} \frac{1-\bar{M}_+^2}{\bar{\rho}_+\bar{q}_+^2}
	\int_0^1 \dot{g}_1^\sharp(\dot{\xi}, z_2) dz_2 \\
	& \quad +\delta \xi^* \left(\kappa \bar{f}_1^+ - \sigma \frac{1}{L-\bar{\xi_*}}
	\int_{\dot{\xi}}^L \Theta (z_1) \right) - \delta \xi^* \left( \sigma \Theta(\dot{\xi}) -
	\sigma \frac{1}{L-\dot{\xi}} \int_{\dot{\xi}}^L \Theta(y_1) dy_1 \right) \\
	& \quad +\delta \xi^* \left( (1-\dot{K}_1) \sigma \Theta(\bar{\xi_*})
	-\kappa \frac{\bar{\rho}_+}{\bar{p}_+} \frac{\bar{p}_-}{\bar{\rho}_-} \bar{f}^-_1
	\right) \\
	& \quad +O(1)\delta \xi^* (\sigma + \kappa)^{\frac{3}{2}}
	+O(1) (\sigma + \kappa)^2
	+O(1) (\sigma + \kappa)(\delta \xi^*)^2 \\
	&= \delta \xi^* \left(-\sigma \dot{K}_1 \Theta(\dot{\xi})
	+ \kappa \dot{K}_2
	+ O(1) (\sigma + \kappa)^{\frac{3}{2}} \right) \\
	& \quad +O(1) (\sigma + \kappa)^2
	+O(1) (\sigma + \kappa)(\delta \xi^*)^2.
	\end{align*}
	This implies
	\begin{align}\label{DI}
	\frac{\partial I}{\partial (\delta \xi^*)} (0,0,0; \dot{U}_-, \dot{c})
	=-\sigma \dot{K}_1 \Theta(\dot{\xi})
	+ \kappa \dot{K}_2
	+ O(1) (\sigma + \kappa)^{\frac{3}{2}}.
	\end{align}
	If we choose $\sigma$ and $\kappa$ sufficiently small and $-\sigma \dot{K}_1 \Theta(\dot{\xi})
	+ \kappa \dot{K}_2 \neq 0$,
	the existence of the solution $\delta \xi^*$ to the equation $I(\delta \xi^*, \delta U, \delta \psi', \delta U_-, 0)=0$ can be proved by the implicit function theorem. Furthermore,  we have
	\begin{align*}
	|\delta \xi^*| &\leq \frac{C}{\left| -\sigma \dot{K}_1 \Theta(\dot{\xi})
		+ \kappa \dot{K}_2 \right|}
	(\sigma+\kappa)^2
	\leq \frac{C}{\beta_0} (\sigma+\kappa),
	\end{align*}
	which complete the proof of the lemma.
\end{proof}

With the solution $\delta \xi^*$, the solvability of the elliptic system \eqref{pt1}-\eqref{pt6} holds so that we can solve the solution
$(\delta p^*, \delta \theta^*)$ and get the estimate
\begin{align}\label{est-pt}
\begin{split}
\|\delta p^*\|_{W_q^1(\Omega)}+\|\delta \theta^*\|_{W_q^1(\Omega)}
&\leq C\big(\|\tilde{f}_1\|_{L^{\beta}(\Omega)}+\|\tilde{f}_2\|_{L^{\beta}(\Omega)}
+\|\delta \Theta_4\|_{W_{\beta}^{1-\frac{1}{\beta}}(\Gamma_4)} \\
&\quad + \|\delta P\|_{W_{\beta}^{1-\frac{1}{\beta}}(\Gamma_3)}
+\|g_1^\sharp\|_{W_{\beta}^{1-\frac{1}{\beta}}(\Gamma_s)}\big).
\end{split}
\end{align}

To complete the iteration of $(\delta p, \delta \theta)$, we need to give an estimate of
$(\delta p^*-\delta \dot{p}_+, \delta \theta^*-\delta \dot{\theta}_+)$.
\begin{thm}
	Suppose $(\delta U, \delta \psi') \in \mathcal{S}(\displaystyle \frac{1}{2} (\sigma + \kappa)^{3/2}; \dot{U}_+, \dot{\psi}')$ and $\delta \xi^*$ satisfies \eqref{sol}. Then $(\delta p^*-\delta \dot{p}_+, \delta \theta - \delta \dot{\theta}_+)$ satisfies the estimate
	\begin{equation*}
		\|\delta p^*-\delta \dot{p}_+\|_{W_q^1(\Omega)}
		+\|\delta \theta^*- \delta \dot{\theta}_+ \|_{W_q^1(\Omega)}
		\leq \frac{1}{2} (\kappa + \sigma)^{\frac{3}{2}},
	\end{equation*}
as $\sigma$ and $\kappa$ are sufficiently small.
\end{thm}

\begin{proof}
	Taking the difference of the equations \eqref{pt1}-\eqref{pt6} with the corresponding equations of the initial approximating solution. From the estimate \eqref{est-pt} we can obtain
	\begin{align}
	\begin{split}
	&\quad \|\delta p^*-\delta \dot{p}_+\|_{W_q^1(\Omega)}
	+\|\delta \theta^*- \delta \dot{\theta}_+ \|_{W_q^1(\Omega)}\\
	&\leq C\big(\|\tilde{f}_1-\kappa \bar{f}_1^+\|_{L^q(\Omega)}
	+\|\tilde{f}_2\|_{L^{\beta}(\Omega)}
	+\|\delta \Theta_4-\sigma \Theta\|_{W_{\beta}^{1-\frac{1}{\beta}}(\Gamma_4)} \\
	&\quad + \|\delta P-P_\mr{e}\|_{W_{\beta}^{1-\frac{1}{\beta}}(\Gamma_3)}
	+\|g_1^\sharp-\dot{g}_1^{\sharp}\|_{W_{\beta}^{1-\frac{1}{\beta}}(\Gamma_s)}\big).
	\end{split}
	\end{align}
	
	Expand the first term in the right hand side:
	\begin{align*}
	\|\tilde{f}_1-\kappa \bar{f}_1^+\|_{L^\beta(\Omega)}
	&= \left\| \frac{\sin \theta}{\rho q} \partial_{z_1} \theta +
	\left( \frac{\cos \theta}{\rho q} \frac{1-M_+^2}{\rho q^2}
	- \frac{1}{\bar{\rho}_+ \bar{q}_+} \frac{1-\bar{M}_+^2}{\bar{\rho}_+  \bar{q}_+^2} \right)
	\partial_{z_1}p \right\|_{L^\beta(\Omega)} \\
	&\quad + \left\| \kappa (f_1^+ - \bar{f}_1^+) \right\|_{L^\beta(\Omega)}
	+\left\| \frac{L-z_1}{L-\psi(z_2)} \psi'(z_2) \partial_{z_1} \theta \right\|_{L^\beta(\Omega)} \\
	&\quad + \left\| \frac{\delta \xi^* - \int_{z_2}^1 \delta \psi'(s) ds}{L-\psi(z_2)}
	\left( \frac{\cos \theta}{\rho q} \frac{1-M_+^2}{\rho q^2} \partial_{z_1} p
	+\frac{\sin \theta}{\rho q} \partial_{z_1}\theta \right) \right\|_{L^{\beta}(\Omega)}\\
	&\leq C\| \delta U \|_{L^\infty} \| (\delta \theta, \delta p) \|_{W_q^1(\Omega)}
	+C \kappa \| \delta U \|_{L^\beta} \\
	&\quad + C(\dot{\xi}) |\delta \xi^*| \| (\delta \theta, \delta p) \|_{W_q^1(\Omega)}
	+ C(\dot{\xi}) \| \delta \psi' \|_{L^\infty} \| (\delta \theta, \delta p) \|_{W_q^1(\Omega)} \\
	&\leq C (\sigma + \kappa)^2.
	\end{align*}
	Similarly, it holds that $ \|\tilde{f}_2\|_{L^{\beta}(\Omega)} \leq C (\sigma + \kappa)^2 $.
	
	In the proof of Theorem \ref{sol thm}, we have shown that
	\begin{align*}
	\begin{split}
	\|g_1^\sharp-\dot{g}_1^{\sharp}\|_{W_{\beta}^{1-\frac{1}{\beta}}(\Gamma_s)}
	&\leq \| \dot{g}_1^\sharp(\dot{\xi}, \cdot)
	-\dot{g}_1^{\sharp} ( \dot{\xi} + \delta \xi^*, \cdot )\|_{W_{\beta}^{1-\frac{1}{\beta}}(\Gamma_s)}
	+ O(1) (\kappa + \sigma)^2 \\
	&\leq C |\delta \xi^*| \| \partial_{z_1}\dot{U}_- \|_{C^{1,\alpha}(\Omega_L)}
	+ O(1) (\kappa + \sigma)^2 \\
	&\leq C (\sigma + \kappa)^2,
	\end{split}\\
	\begin{split}
	\|\delta \Theta_4-\sigma \Theta\|_{W_{\beta}^{1-\frac{1}{\beta}}(\Gamma_4)}
	&\leq \sigma \| \Theta(z_1+\frac{L-z_1}{L-\xi_*} \delta \xi^*) - \Theta(z_1) \|
	_{W_{\beta}^{1-\frac{1}{\beta}}(\Gamma_4)} \\
	&\leq C \sigma \| \Theta \|_{C^{1,\alpha}(\Gamma_s)}
	| \delta \xi^* | \\
	&\leq  C \sigma (\kappa + \sigma),
	\end{split}\\
	\|\delta P-P_\mr{e}\|_{W_{\beta}^{1-\frac{1}{\beta}}(\Gamma_3)}
	&\leq C (\kappa + \sigma)^2.
	\end{align*}
	
	Combining the above estimate, we obtain that there exists a constant $C_1$ such that
	\begin{align*}
	\|\delta p^*-\delta \dot{p}_+\|_{W_q^1(\Omega)}
	+\|\delta \theta^*- \delta \dot{\theta}_+ \|_{W_q^1(\Omega)}
	\leq C_1 (\kappa + \sigma)^2 \leq \frac{1}{2} (\kappa + \sigma)^{\frac{3}{2}},
	\end{align*}
	for sufficiently small $\kappa$ and $\sigma$.
\end{proof}

\subsubsection{Solve for $\delta S^*$ and $\delta Z^*$}
Consider the equations
\begin{align*}
&\partial_{z_1} (\delta S^*) = \tilde{f}_4 (\delta U; \delta \psi', \delta \xi^*),
& &\text{in $\tilde{\Omega}$}, \\
&\delta S^*(\dot{\xi}, z_2) = g_3^\sharp  (\delta U, \delta U_-; \delta \psi', \delta \xi^*),
& &\text{on $\tilde{\Gamma}_s$},
\end{align*}
and
\begin{align*}
&\partial_{z_1} (\delta Z^*) = \tilde{f}_5 (\delta U; \delta \psi', \delta \xi^*),
& & \text{in $\tilde{\Omega}$},\\
&\delta Z^*(\dot{\xi}, z_2) = g_4^\sharp (\delta U, \delta U_-; \delta \psi', \delta \xi^*),
& &\text{on $\tilde{\Gamma}_s$},
\end{align*}
which gives the solution $(\delta S^*, \delta Z^*)$:
\begin{align}
&\delta S^*(z_1, z_2)=g_3^\sharp(z_2) + \int_{\dot{\xi}}^{z_1} \tilde{f}_4(s,z_2) ds,\\
&\delta Z^*(z_1, z_2)=g_4^\sharp(z_2) + \int_{\dot{\xi}}^{z_1} \tilde{f}_5(s,z_2) ds.
\end{align}
Similar to the estimate of $\|g_1^\sharp-\dot{g}_1^{\sharp}\|_{W_{\beta}^{1-\frac{1}{\beta}}(\Gamma_s)}$, we obtain
\begin{align*}
\| \delta S^* - \delta \dot{S}_+  \|_{W_{\beta}^{1-\frac{1}{\beta}}(\Gamma_s)}
&\leq \| \dot{g}_3^\sharp(\dot{\xi}, \cdot)
-\dot{g}_3^{\sharp} ( \dot{\xi} + \delta \xi^*, \cdot )\|_{W_{\beta}^{1-\frac{1}{\beta}}(\Gamma_s)}
+ O(1) (\kappa + \sigma)^2\\
&\leq C(\kappa + \sigma)^2 \leq \frac{1}{2} (\kappa + \sigma)^{\frac{3}{2}},
\end{align*}
and
\begin{align*}
\| \delta S^* - \delta \dot{S}_+  \|_{C^0(\Omega)}
&\leq \| \dot{g}_3^\sharp(\dot{\xi}, \cdot)
-\dot{g}_3^{\sharp} ( \dot{\xi} + \delta \xi^*, \cdot )\|_{W_{\beta}^{1-\frac{1}{\beta}}(\Gamma_s)} \\
&\quad +\kappa \| \int_{\xi^*}^{z_1} (\tilde{f}_4-\bar{f}_4^+) (s,z_2) ds  \|_{C^0(\Omega)} \\
&\quad +\| \frac{\dot{\xi}-\psi(z_2)}{L-\psi(z_2)} \int_{\xi^*}^{z_1} \partial_{z_1}(\delta S) (s,z_2) ds  \|_{C^0(\Omega)}\\
&\leq C(\kappa +\sigma)^2
+ C \kappa \|\delta U\|_{W^1_q(\Omega)}
+ C |\delta \xi^*| \| \delta S \|_{C^0(\Omega)}\\
&\leq C(\kappa + \sigma)^2 \leq \frac{1}{2} (\kappa + \sigma)^{\frac{3}{2}},
\end{align*}
for sufficiently small $\kappa$ and $\sigma$.

Using the same estimate, we can also get
\begin{align*}
\| \delta Z^* - \delta \dot{Z}_+  \|_{W_{\beta}^{1-\frac{1}{\beta}}(\Gamma_s)}
+ \| \delta Z^* - \delta \dot{Z}_+  \|_{C^0(\Omega)}
\leq C(\kappa + \sigma)^2 \leq \frac{1}{2} (\kappa + \sigma)^{\frac{3}{2}},
\end{align*}
for sufficiently small $\kappa$ and $\sigma$.

\subsubsection{Solve for $\delta q^*$}
We will use the equations of the divergence form to solve $\delta q^*$:
\begin{align*}
&\partial_{z_1} (\bar{q}_+ \delta q^* + \frac{1}{\bar{\rho}_+} \delta p^* + \bar{T}_+ \delta S^* +q_{\mr{e}} \delta Z^*)
= \partial_{z_1} \tilde{f}_3 (\delta U; \delta \psi', \delta \xi^*),
& &\text{in $\tilde{\Omega}$}, \\
&\delta q^*(\dot{\xi}, z_2) = g_2^\sharp  (\delta U, \delta U_-; \delta \psi', \delta \xi^*),
& &\text{on $\tilde{\Gamma}_s$}.
\end{align*}
Integrating the equations from $\dot{\xi}$ to 1, we have
\begin{align*}
&\quad (\bar{q}_+ \delta q^* + \frac{1}{\bar{\rho}_+} \delta p^* + \bar{T}_+ \delta S^* +q_{\mr{e}} \delta Z^*) (z_1, z_2) \\
&= (\bar{q}_+ \delta q^* + \frac{1}{\bar{\rho}_+} \delta p^* + \bar{T}_+ \delta S^* +q_{\mr{e}} \delta Z^*)
(\dot{\xi}, z_2)
- \tilde{f}_3 (\dot{\xi} , z_2) + \tilde{f}_3 (z_1, z_2).
\end{align*}
Since
\begin{align*}
\tilde{f}_3 (\delta U) + \Phi(\bar{U}_+)
&= (\bar{q}_+ \delta q^* + \frac{1}{\bar{\rho}_+} \delta p^* + \bar{T}_+ \delta S^* +\beta \delta Z^*) - \Phi(U) + \Phi(\bar{U}_+) \\
&= -\int_0^1 D_u^2 \Phi(\bar{U}_+ + t\delta U) dt (\delta U)^2,
\end{align*}
we can obtain the estimate
\begin{align*}
\|\tilde{f}_3 (\delta U) + \Phi(\bar{U}_+) \|_{C^0(\Omega)} \leq C (\kappa + \sigma)^2 .
\end{align*}
Hence $\delta q^*- \delta \dot{q}_+$ can be estimated as
\begin{align*}
&\quad \| \delta q^* - \delta \dot{q}_+  \|_{W_{\beta}^{1-\frac{1}{\beta}}(\Gamma_s)}
+ \| \delta q^* - \delta \dot{q}_+  \|_{C^0(\Omega)} \\
&\leq C \bigg( \sum_{i=1}^4 \| \dot{g}_i^\sharp (\dot{\xi}, \cdot)
-\dot{g}_i^{\sharp} ( \dot{\xi} + \delta \xi^*, \cdot )\|_{W_{\beta}^{1-\frac{1}{\beta}}(\Gamma_s)}\\
&\quad + \| (\delta p^*-\delta \dot{p}_+, \delta Z^*-\delta \dot{Z}_+, \delta S^*-\delta \dot{S}_+) \|_{C^0(\Omega)} \\
&\quad + \|\tilde{f}_3 (\delta U) + \Phi(\bar{U}_+) \|_{C^0(\Omega)}\bigg) \\
&\leq C (\kappa + \sigma)^2 \leq \frac{1}{2} (\kappa + \sigma)^{\frac{3}{2}}.
\end{align*}

\subsubsection{Update the shock front $(\delta \psi^*)'$}
From \eqref{its2}, we can update $(\delta \psi^*)'$ by
\begin{align*}
(\delta \psi^*)' = \frac{1}{[p]} (\beta_4^+ \delta U^* - g_5).
\end{align*}
Hence we have
\begin{align*}
\| (\delta \psi^*)' \|_{W_{\beta}^{1-\frac{1}{\beta}}(\Gamma_s)}
&\leq C( \| \delta U^* \|_{W_{\beta}^{1-\frac{1}{\beta}}(\Gamma_s)}
+\| g_5 \|_{W_{\beta}^{1-\frac{1}{\beta}}(\Gamma_s)} ) \\
\| (\delta \psi^*)' - (\delta \dot{\psi})' \|_{W_{\beta}^{1-\frac{1}{\beta}}(\Gamma_s)}
&\leq C( \| \delta U^* - \delta \dot{U}_+ \|_{W_{\beta}^{1-\frac{1}{\beta}}(\Gamma_s)}
+ \| g_5 - \dot{g}_5 \|_{W_{\beta}^{1-\frac{1}{\beta}}(\Gamma_s)} ).
\end{align*}
Similar to the computation of $g_i, i=1,2,3,4$,
\begin{align*}
g_5=-\beta_5^- \dot{U}_-(\dot{\xi} + \delta \xi^*, z_2) + O(1) (\kappa + \sigma)^2.
\end{align*}
Therefore
\begin{align*}
\| (\delta \psi^*)' -  (\delta \dot{\psi})' \|_ {W_{\beta}^{1-\frac{1}{\beta}}(\Gamma_s)}
&\leq C |\delta \xi^*| \| \partial_{z_1} U^- \|_{C^{1,\alpha}(\Omega_L)}
+ C(\kappa + \sigma)^2 \\
&\leq C(\kappa + \sigma)^2 \leq \frac{1}{2} (\kappa + \sigma)^{\frac{3}{2}}.
\end{align*}

\subsubsection{The contraction of the iteration mapping}

Combining the above four steps, we have proved that the iteration mapping
$	J' : \mathcal{S}(\displaystyle \frac{1}{2} (\sigma + \kappa)^{3/2}; \dot{U}_+, \dot{\psi}')
\to \mathcal{S}(\displaystyle \frac{1}{2} (\sigma + \kappa)^{3/2}; \dot{U}_+, \dot{\psi}')$ which maps $(\delta U; \delta \psi')$ to
$(\delta U^*; (\delta \psi^*)') $ is well defined for sufficiently small $\kappa$ and $\sigma$ . To find the fixed point of the iteration, we will prove that the iteration mapping
$J'$ is a contractive mapping.

\begin{thm}
	There exists small constants $\sigma_0, \kappa_0$ such that for any $\sigma \in (0,\sigma_0)$ and $\kappa \in (0,\kappa_0)$, the iteration mapping
	$	J' : \mathcal{S}(\displaystyle \frac{1}{2} (\sigma + \kappa)^{3/2}; \dot{U}_+, \dot{\psi}')
	\to \mathcal{S}(\displaystyle \frac{1}{2} (\sigma + \kappa)^{3/2}; \dot{U}_+, \dot{\psi}')$ is contractive.
\end{thm}

\begin{proof}
	Assume that $(\delta U_k; \delta \psi'_k)
	\in  \mathcal{S}(\displaystyle \frac{1}{2} (\sigma + \kappa)^{3/2}; \dot{U}_+, \dot{\psi}'), k=1,2 $.
	According to the discussion in the previous sections, there exists $\delta \xi^*_k$ satisfying \eqref{sol} and
	\begin{align*}
	(\delta U^*_k; (\delta \psi^*_k)')
	=J'(\delta U_k; \delta \psi'_k) \in  \mathcal{S}(\displaystyle \frac{1}{2} (\sigma + \kappa)^{3/2}; \dot{U}_+, \dot{\psi}'), k=1,2.
	\end{align*}
	To estimate $\| \delta U^*_1 - \delta U^*_2 \|_{(\Omega, \Gamma_s)}$ and
	$ \| (\delta \psi^*_1)' -(\delta \psi^*_1)' \|_{W_{\beta}^{1-\frac{1}{\beta}}(\Gamma_s)} $,
	we need to estimate $| \delta \xi^*_2 - \delta \xi^*_1 |$ first.
	
	Since $\delta \xi^*_1, \delta \xi^*_2$ both satisfy the solvability condition \eqref{sol}, it implies that
	\begin{align*}
	0&= I(\delta \xi_2^*; \delta U_2, \delta \psi'_2; \delta U_-, 0)
	- I(\delta \xi_1^*; \delta U_1, \delta \psi'_1; \delta U_-, 0) \\
	&= \int_0^1 \frac{\partial I}{\partial (\delta \xi^*)}
	(\delta \xi^*_t; \delta U_2, \delta \psi_2'; \delta U_-, 0)dt \cdot
	(\delta \xi^*_2 - \delta \xi^*_1) \\
	& \quad + \int_0^1 \nabla_{(\delta U, \delta \psi')}
	I(\delta \xi^*_1; \delta U_t, \delta \psi_t'; \delta U_-, 0)dt \cdot
	(\delta U^*_1 - \delta U^*_2, (\delta \psi^*_1)' -(\delta \psi^*_1)'),
	\end{align*}
	where
	$\delta \xi^*_t = t \delta \xi^*_2 + (1-t) \delta \xi^*_1$,
	$\delta U_t = t \delta U_2 + (1-t) \delta U_1$ and
	$\delta \psi_t' = t \delta \psi_2' + (1-t) \delta \psi_1'$.
	
	Similar to the computation of \eqref{DI},
	\begin{align*}
	&\quad \frac{\partial I}{\partial (\delta \xi^*)}
	(\delta \xi^*_t; \delta U_2, \delta \psi_2'; \delta U_-, 0)\\
	&=\frac{\partial I}{\partial (\delta \xi^*)}
	(0; 0, 0; \delta \dot{U}_-, \dot{c})
	+ \int_0^1 \nabla_{(\delta \xi^*, ; \delta U, \delta \psi'; \delta U_-, c)}
	\frac{\partial I}{\partial (\delta \xi^*)}
	(s\delta \xi^*_t; s\delta U_2, s\delta \psi_2'; \delta U_-^s, (1-s)\dot{c})ds \\
	&= (-\sigma \dot{k} \Theta(\dot{\xi})
	+ \kappa (\bar{f}_1^+ - \frac{\bar{\rho}_+}{\bar{p}_+} \frac{\bar{p}_-}{\bar{\rho}_-} \bar{f}^-_1)
	+ O(1) (\sigma + \kappa)^{\frac{3}{2}}) + O(1) (\sigma + \kappa)^2,
	\end{align*}
	and
	\begin{align*}
	\nabla_{(\delta U, \delta \psi')}
	I(\delta \xi^*_1; \delta U_t, \delta \psi_t'; \delta U_-, 0) = O(1) (\sigma+\kappa).
	\end{align*}
	Hence as $\sigma$ and $\kappa$ are sufficiently small and \eqref{admissble condition} holds,
	\begin{align}\label{xi_control}
	| \delta \xi^*_2 - \delta \xi^*_1 |
	\leq C\left( \| \delta U^*_1 - \delta U^*_2 \|_{(\Omega, \Gamma_s)}
	+ \| (\delta \psi^*_1)' -(\delta \psi^*_1)' \|_{W_{\beta}^{1-\frac{1}{\beta}}(\Gamma_s)} \right),
	\end{align}
	where $C$ depends on $\beta_0$ and the background solution.
	
	Substituting $(\delta U_k^*, (\delta \psi^*_k)')$ and $(\delta U_k, \delta \psi_k')$ into \eqref{it1}-\eqref{its2} for $k=1,2$ and taking the difference, through an analogous computation with the help of \eqref{xi_control}, we can get the estimate
	\begin{align*}
	&\quad \| \delta U^*_1 - \delta U^*_2 \|_{(\Omega, \Gamma_s)}
	+ \| (\delta \psi^*_1)' -(\delta \psi^*_1)' \|_{W_{\beta}^{1-\frac{1}{\beta}}(\Gamma_s)}\\
	&\leq O((\sigma+ \kappa)^{\frac{1}{2}})
	\left(\| \delta U_1 - \delta U_2 \|_{(\Omega, \Gamma_s)}
	+ \| (\delta \psi_1)' -(\delta \psi_1)' \|_{W_{\beta}^{1-\frac{1}{\beta}}(\Gamma_s)}\right)\\
	&\leq \frac{1}{2}
	\left(\| \delta U_1 - \delta U_2 \|_{(\Omega, \Gamma_s)}
	+ \| (\delta \psi_1)' -(\delta \psi_1)' \|_{W_{\beta}^{1-\frac{1}{\beta}}(\Gamma_s)}\right),
	\end{align*}
	for sufficiently small $\sigma$ and $\kappa$.
	
	Therefore, $J'$ is a contractive map.
\end{proof}

Using the contraction mapping theorem, the iteration mapping $J'$ has a unique fixed point $(\delta U, \delta \psi')$ which is a solution to Problem \ref{problem: Lagrange}  and
has the estimate
\begin{align*}
\quad \| \delta U_+ - \dot{U}_+ \|_{(\Omega, \Gamma_s)}
+ \| \delta \psi' -\dot{\psi}' \|_{W_{\beta}^{1-\frac{1}{\beta}}(\Gamma_s)}
\leq \frac{1}{2} (\kappa + \sigma)^{\frac{3}{2}}.
\end{align*}
Therefore Theorem \ref{Main Theorem} is proved.

\section*{Acknowlegements}
The research of the paper was supported by Natural Science Foundation of China under Grant Nos. 11971308 and 11631008. The research of Qin Zhao was also supported in part by Natural Science Foundation of China under Grant Nos. 12101471.

\end{document}